\documentclass[12pt]{article}

\usepackage{tikz}
\usetikzlibrary{calc}
\usetikzlibrary{positioning,backgrounds}

\usepackage{adjustbox}

\usepackage{amsmath,amsthm,amsfonts,amssymb, color,xcolor,subcaption,graphicx} 

\usepackage{todonotes, dsfont}

\usepackage{tikz-cd}

\newtheorem{theorem}{Theorem}[section]

\newtheorem{example}[theorem]{Example}

\newtheorem{lemma}[theorem]{Lemma}
\newtheorem{definition}[theorem]{Definition}

\def\cB{\mathcal{B}}

\def\cF{\mathcal{F}}

\def\cH{\mathcal{H}}

\def\cL{\mathcal{L}}

\def\cN{\mathcal{N}}
\def\cP{\mathcal{P}}

\def\cS{\mathcal{S}}

\def\cZ{\mathcal{Z}}

\def\bE{\mathbb{E}}

\def\bP{\mathbb{P}}
\def\bR{\mathbb{R}}

\def\bZ{\mathbb{Z}}

\newcommand{\fH}{\mathfrak{H}}

\def\k{\kappa}

\def\wt{\widetilde{w}}

\newcommand{\fm}{\mathfrak{m}}

\newcommand{\blackbox}[1]{  {\color{black}\boxed{\color{black}#1}}}

\begin{document}

\title{Chaotic CLT using 
Poisson Fourth Moment Theorem \\ and Graph Theory}

\author{Raluca M. Balan\footnote{Corresponding author. University of Ottawa, Department of Mathematics and Statistics, 150 Louis Pasteur Private, Ottawa, Ontario, K1N 6N5, Canada. E-mail address: rbalan@uottawa.ca.} \footnote{Research supported by a grant from the Natural Sciences and Engineering Research Council of Canada.}
\and
Jaime Garza \footnote{University of Ottawa, Department of Mathematics and Statistics, 150 Louis Pasteur Private, Ottawa, Ontario, K1N 6N5, Canada. E-mail address: jgarza@uottawa.ca.}
}

\date{October 19, 2026}
\maketitle

\begin{abstract}
\noindent {In this article, we consider the parabolic Anderson model and the hyperbolic Anderson model in dimension 1, driven by a {\em time-independent} L\'evy noise, which is either white, or colored with integrable coloration kernel $\k$. Using a Poisson version of the Fourth Moment Theorem due to \cite{DVZ18}, we prove that the spatial integral of the solution satisfies a central limit theorem. We prove that this result also holds for the hyperbolic Anderson model in dimension $d=2$, if in addition $\k$ is bounded. To estimate the fourth cumulant of the projection of the spatial integral on a fixed chaos, we develop a novel technique using graph theory, which allows us to handle some intricate integrals involving the modified contractions.}
\end{abstract}

\noindent {\em MSC 2020:} Primary 60H15; Secondary 60G60, 60G51

\vspace{1mm}

\noindent {\em Keywords:} stochastic partial differential equations, random fields, Malliavin calculus, normal approximations, Poisson random measure, L\'evy noise

\pagebreak


\section{Introduction}

In this article, we consider the parabolic Anderson model {in dimension $d=1$}:
\begin{align}
\label{PAM}
	\begin{cases}
		\dfrac{\partial u}{\partial t} (t,x)
		=  \dfrac{1}{2}\dfrac{\partial^2 u}{\partial x^2} (t,x)+ u(t,x) \dot{X}(x), \quad
	t>0, \ x \in \bR \\
		u(0,x) = 1,  \quad x \in \bR.
	\end{cases}
\end{align}
and the hyperbolic Anderson model {in dimension $d \in \{1,2\}$}: 
\begin{align}
\label{HAM}
	\begin{cases}
		\dfrac{\partial^2 u}{\partial^2 t} (t,x)
		=  
\Delta u(t,x) + u(t,x) \dot{X}(x), \quad
	t>0, \ x \in \bR^d \\
		u(0,x) = 1, \ \dfrac{\partial u}{\partial t} (0,x)=0, \quad x \in \bR^d,
	\end{cases}
\end{align}
driven by the (time-independent) L\'evy colored noise introduced in the recent preprint \cite{BW26}. The goal is to prove a central limit theorem (as $R \to \infty$) for the {\em spatial average} of the solution:
\begin{equation}
\label{def-FR}
F_R(t)=\int_{B_R}\big(u(t,x)-1\big)dx,
\end{equation}
where $B_R=\{x\in \bR^d;|x|\leq R\}$ and $|\cdot|$ is the Euclidean norm. We let $\sigma_R^2(t)={\rm Var}(F_R(t))$.

\medskip

{

This problem has received a lot of attention in the recent years, especially in the case of the Gaussian noise, starting with the seminal article \cite{HNV20}. We refer the reader to \cite{BNZ,DNZ20,HNVZ20,NSZ21,NXZ22,NZ22} for a sample of relevant references. These references rely on a combination of Malliavin calculus with Stein's method, which yields the precise rate of convergence in the central limit theorem, a result which is now called the Quantitative Central Limit Theorem (QCLT). The major tool is a key estimate for the $p$-th moment of the Malliavin derivative of the solution. In the Gaussian case, the time-independent noise can be viewed as a ``black'' noise, being the opposite of the white noise: if the noise is a fractional Brownian motion in time with index $H \in (0,1)$, the case $H=1/2$ corresponds formally to the white noise, while $H=1$ corresponds to the time-independent noise.

\medskip

Recently, the QCLT problem has been studied also for the spatial average of the solution of a stochastic partial differential equation (SPDE) driven by a finite-variance L\'evy noise, using Malliavin calculus with respect to the underlying Poisson random measure of the noise. References \cite{BZ24,BZ26} studied the case of the L\'evy white noise, while \cite{BS26} treated the L\'evy colored noise.
In the present article, we continue these investigations in the case of the time-independent L\'evy noise, which may be white or colored in space. This case introduces new challenges which we explain below.

First, defining an It\^o integral with respect to the time-independent noise (Gaussian or L\'evy) is a delicate issue. In the recent article \cite{BG26}, in dimension $d=1$, we introduced an It\^o integral with respect to a (finite-variance) time-independent L\'evy noise $L$, using a spatially-indexed filtration $(\cF_x)_{x\in \bR}$. But this integral cannot be used for defining the mild solution of the SPDE. Consider for instance the equation $\cL u(t,x)=u(t,x)\dot{L}(x)$ with $u(0,x)=1$, where $\cL$ is a differential operator. One cannot define recursively the Picard iteration sequence 
\[
u_{n+1}(t,x)=1+\int_0^t \int_{\bR}G_{t-s}(x-y) u_n(s,y)L(dy)ds \quad  (u_0=1),
\]
since the integrand is not $\cF_y$-measurable.
For instance, in the case of the wave equation, 
$u_1(t,x)=1+\frac{1}{2} L\big((x-t+s,x+t-s)\big)$ is $\cF_{x+t-s}$-measurable. Therefore, the only integral that we can consider for defining the solution is the Skorohod integral. This is reviewed in Section \ref{section-prelim} below.

Secondly, the traditional methods for proving a QCLT in the Gaussian case do not work in the case of the time-independent L\'evy noise. The major problem is that, in the absence of hypercontractivity and Burkholder-Davis-Gundy inequality, there are no techniques for handling $p$-th moments of stochastic integrals. The recent article \cite{tara} gives a powerful version of the second-order Poincar\'e inequality for the Poisson case, which has been applied in \cite{BZ24} for proving the QCLT for the wave equation driven by L\'evy white noise. But applying this inequality requires an estimate for the $p$-th moments of the Malliavin derivative of the solution, which is not available in the time-independent case.
}

\medskip

We continue now by describing our framework. We start with the definition of the noise.
We let $\cS(\bR^d)$ be the space of rapidly decreasing functions on $\bR^d$, and $\cS'(\bR^d)$ be its dual, the space of tempered distributions.
The (time-independent) {\em L\'evy colored noise} is the collection $X=\{X(\varphi);\varphi\in \cS(\bR^d)\}$ of random variables, given by:
\[
X(\varphi)=L(\varphi *\k), 
\]
where $\k$ is a ``coloration'' kernel,
and $L$ is a (time-independent) {\em L\'evy white noise} given by:
\begin{equation}
\label{def-L}
L(\phi)=\int_{\bR^d \times \bR_0}\phi(x) z \widehat{N}(dx,dz) \quad \mbox{for all $\phi \in L^2(\bR^d)$}.
\end{equation}
Here 
$N$ is a Poisson random measure on $\bR^d \times \bR_0$ of intensity $\fm={\rm Leb} \times \nu$,
defined on a complete probability space $(\Omega,\cF,\bP)$, and $\widehat{N}(F)=N(F)-\fm(F)$ is the compensated version of $N$. The space $\bR_0:=\bR \verb2\2 \{0\}$ is endowed with the distance $d(x,y)=|x^{-1}-y^{-1}|$.

{
In this article, we assume that $\k \in L^1(\bR^d)$, or $d=1$ and $\k=\delta_0$ (i.e. $X=L$).}
Moreover, the measure $\nu$ satisfies the condition:
\begin{equation}
\label{def-m2}
m_2:=\int_{\bR_0}|z|^2 \nu(dz)<\infty.
\end{equation}

\medskip

{
If $\k \in L^1(\bR^d)$, $X(\varphi)$ is well-defined for any $\varphi \in L^2(\bR^d)$, since 
$\varphi * \k \in L^2(\bR^d)$.} Letting
$f=\k*\widetilde{\k}$, where $\widetilde{\k}(x)=\k(-x)$ for all $x \in \bR^d$, we have:
for any $\varphi,\psi \in \cS(\bR^d)$, 
\begin{equation}
\label{def-prod0}
\langle \varphi,\psi\rangle_0:=\int_{\bR^d} \int_{\bR^d}\varphi(x)\psi(y)f(x-y)dxdy=\int_{\bR^d}\cF \varphi(\xi) \overline{\cF \psi(\xi)}\mu(d\xi),
\end{equation}
where $\mu(d\xi)=\frac{1}{(2\pi)^d}|\cF \k(\xi)|^2 d\xi$.
Here
$\cF \varphi(x)=\int_{\bR^d} e^{-i \xi \cdot x} \varphi(x)dx$ is the Fourier transform of function $\varphi \in L^1(\bR^d)$, and $\xi \cdot x$ is the inner product in $\bR^d$.
For any $\varphi,\psi \in \cS(\bR^d)$, 
\[
\bE[X(\varphi)X(\psi)]=m_2 \langle \varphi *\k, \psi*\k\rangle_{L^2(\bR^d)}=m_2 \langle \varphi,\psi\rangle_0.
\]

\medskip

Intuitively, the solution of \eqref{PAM} should satisfy:
\begin{equation}
\label{HAM-eq}
u(t,x) =1+\int_0^t \int_{\bR^d} G_{t-s}(x-y) u(s,y)X(dy)ds,
\end{equation}
where $G_t$ is the fundamental solution of the heat equation:
\[
G_t(x)=\frac{1}{(2\pi t)^{d/2}} \exp\left(-\frac{|x|^2}{2t} \right),
\]
and $|\cdot|$ is the Euclidean norm. Similarly, the solution of \eqref{HAM} should satisfy \eqref{HAM-eq}, where $G_t$ is the fundamental solution of the wave equation:
\begin{align*}
G_t(x)&=\frac{1}{2}1_{\{|x|<t\}} \quad \mbox{if $d=1$},\\
G_t(x)&=\frac{1}{2\pi} \frac{1}{\sqrt{t^2-|x|^2}} 1_{\{|x|<t\}} \quad \mbox{if $d=2$}.
\end{align*}

The rigorous definition of the solution is given Section \ref{section-prelim} below. From \cite{BW26}, it is known that the solutions of \eqref{PAM} and \eqref{HAM} exist, provided that {\em Dalang's condition} holds, i.e.
\begin{equation}
\label{dalang}
\int_{\bR^d}\frac{1}{1+|\xi|^2}\mu(d\xi)<\infty.
\end{equation}
Note that condition \eqref{dalang} always holds in dimension $d=1$. 

\medskip

We are now ready to state our main results. 

\begin{theorem}
\label{cov-th}
{
Suppose that either one of the following conditions hold: \\
(i) $d=1$, $k=\delta_0$ (i.e. $X=L$), and $u$ is the solution of \eqref{PAM} or \eqref{HAM}; \\
(ii) $d=1$, $k \in L^1(\bR)$,  and $u$ is the solution of \eqref{PAM} or \eqref{HAM}; \\
(iii) $d=2$, $\k \in L^1(\bR^2)$, \eqref{dalang} holds, and $u$ is the solution of \eqref{HAM}.
}

Then $\{u(t,x)\}_{x\in \bR^d}$ is strictly stationary and 
\begin{equation}
\label{def-K}
\lim_{R \to \infty}\frac{1}{R^d}\bE[F_R(t)F_R(s)]  =: c(t,s) \quad \mbox{is finite}.
\end{equation}
In particular, $\sigma_R^2(t) \sim c(t) R^d$ as $R \to \infty$, where $c(t)=c(t,t)$.
\end{theorem}

\begin{theorem}
\label{CLT}
{
Suppose that either one of conditions (i)-(iii) from Theorem \ref{cov-th} hold. In case (iii), we assume in addition that $\k \in L^{\infty}(\bR^2)$.}
Then 
\[
\frac{F_R(t)}{\sigma_R(t)} \stackrel{d}{\longrightarrow} Z \sim \cN(0,1) \quad \mbox{as  $R \to \infty$},
\]
where
 $\stackrel{d}{\longrightarrow}$ denotes the convergence in distribution.
\end{theorem}



{

Theorem \ref{CLT} is called a ``chaotic CLT'', because it is proved by handling separately the projections of $F_R(t)$ on each Poisson chaos. Using the Poisson version of the Fourth Moment Theorem (due to \cite{DVZ18}), we prove that the projection $I_n(g_{n,R})$ on the $n$-th chaos has an asymptotically normal distribution as $R \to \infty$. Lemma \ref{cumulant} shows that the fourth cumulant of $I_n(g_{n,R})$ is bounded by a linear combination of the squared norms of the contractions
$g_{n,R}*_{r}^{\ell} g_{n,R}$, which are difficult to work with. The novelty of this article is that we devise a graph-theoretic technique for estimating the norms of these contractions.

\medskip

This article is organized as follows. In Section \ref{section-prelim}, we introduce some preliminaries. Sections \ref{section-limit} and \ref{section-CLT} contain the proofs of Theorems \ref{cov-th}, respectively \ref{CLT}. The appendices contain some auxiliary results.}

\section{Preliminaries}
\label{section-prelim}

In this section, we include some preliminaries regarding the existence ofsolutions of equations \eqref{PAM} and \eqref{HAM}.

\medskip

Using the relations between $(X,L)$ and $(L,\widehat{N})$, we can write formally equation \eqref{HAM-eq} as:
\begin{align*}
u(t,x) 
&=1+\int_0^t \int_{\bR^d} \big(G_{t-s}(x-\cdot) u(s,\cdot)*\k \big)(y)L(dy)ds\\
&=1+\int_0^t \int_{\bR^d} \int_{\bR_0}\big(G_{t-s}(x-\cdot) u(s,\cdot)*\k \big)(y) z \widehat{N}(dy,dz)ds
\end{align*}

But there is no stochastic integral with respect to $\widehat{N}$ for {\em random} integrands. Since $N$ does not have a time component, there is no filtration $(\cF_t)_{t\geq 0}$ indexed by time, and one cannot introduce a suitable concept of a ``predictable'' random field, which could lead to a proper definition of Picard iterations, as in the study of classical SPDEs with time-dependent noise. Therefore, to define the solution rigorously, we use Malliavin calculus. 

\subsection{Malliavin calculus on the Poisson space}

In this section, we include some basic material about Malliavin calculus with respect to the compensated Poisson random measure $\widehat{N}$.

We recall that $\cH=L^2(Z,\cZ,\fm)$, where 
\begin{equation*}
(Z,\cZ,\fm)=\big(\bR^d \times \bR_0,\ \cB(\bR^d) \otimes \cB(\bR_0), \ {\rm Leb} \times \nu\big).
\end{equation*}
Then $\cH^{\otimes n}=L^2(Z^n,\cZ^{n},\fm^n)$. We let $\cH^{\odot n}$ be the set of symmetric functions in $\cH^{\otimes n}$.

\medskip

$\bullet$
Any random variable $F \in L^2(\Omega)$ which is $\cF^N$-measurable has the {\em Poisson-chaos expansion}:
\begin{equation}
\label{Poisson-chaos}
F=\bE(F)+\sum_{n\geq 1}I_n(f_n), \quad \mbox{for some $f_n \in \cH^{\odot n}$},
\end{equation}
where $I_n$ is the multiple integral with respect to $\widehat{N}$, and the series is orthogonal in $L^2(\Omega)$. 
For any $f \in \fH^{\otimes n}$, $I_n(f)=I_n(\widetilde{f})$,
\[
\bE[I_n(f)]=0 \quad \mbox{and} \quad \bE|I_n(f)|^2=n! \|\widetilde{f}\|_{\fH^{\otimes n}}^2,
\]
where $\widetilde{f}$ is the symmetrization of $f$, defined by:
\[
\widetilde{f}(\xi_1,\ldots,\xi_n)=\frac{1}{n!}\sum_{\rho \in S_n} f(\xi_{\rho(1)},\ldots,\xi_{\rho(n)}), \quad \mbox{for $\xi_1,\ldots,\xi_n \in {\bf Z}$},
\]
and $S_n$ is the set of permutations of $1,\ldots,n$. 

\medskip
$\bullet$
For any random variable $F \in L^2(\Omega)$ with chaos expansion \eqref{Poisson-chaos}, we define the {\em Malliavin derivative} of $F$ (with respect to $\widehat{N}$) by:
\[
D_{\xi}F=\sum_{n\geq 1}nI_{n-1}\big(f_n(\cdot,\xi)\big), \quad \mbox{for all} \quad \xi \in {\bf Z},
\]
provided that 
\[
\bE\|DF\|_{\fH}^2=\sum_{n\geq 1}nn! \|\widetilde{f}_n\|_{\fH^{\otimes n}}^2<\infty.
\] In this case, we write $F \in {\rm Dom}(D)$.

\medskip

$\bullet$
We denote by $\delta:{\rm Dom}(\delta)\to L^2(\Omega) $ the adjoint of $D$, where ${\rm Dom}(\delta)$ is the set of $V \in L^2(\Omega;\fH)$ for which there exists a constant $C=C_V>0$ depending on $V$, such that 
\[
\big|\bE \langle DF,V \rangle_{\fH}\big| \leq C \|F\|_2 \quad \mbox{for any $F \in {\rm Dom}(D)$}.
\] 
We say that $\delta(V)$ is 
{\em the Skorohod integral} of $V$ with respect to $\widehat{N}$, and we write
\[
\delta(V)=\int_{ \bR^d } \int_{\bR_0}V(x,z)\widehat{N}(\delta x, \delta z).
\]
By duality, for any $V \in {\rm Dom}(\delta)$,
\[
\bE\langle DF,V \rangle_{\fH}=\bE[F \delta(V)] \quad \mbox{for any $F \in {\rm Dom}(D)$}.
\]

\subsection{Existence of solution and chaos expansions}

In this section, we introduce the definition of solution of equation \eqref{PAM} or \eqref{HAM}. We refer the reader to Section 4.2 of \cite{BW26} for the heuristics behind this definition. 

\medskip

We introduce some notation. For any $(t,x) \in \bR_{+}\times \bR^d$ and $x_1, \ldots,x_n\in \bR^d$, we let:
\begin{align}
\nonumber
f_n(t_1,x_1,\ldots,t_n,x_n,t,x)&=G_{t-t_n}(x-x_n)G_{t_n-t_{n-1}}(x_n-x_{n-1}) \ldots G_{t_2-t_1}(x_2-x_1) 1_{T_n(t)} (\pmb{t}) \\
\label{def-fn-x}
f_n(\pmb{x},x;t)&=\int_{T_n(t)} f_n(t_1,x_1,\ldots,t_n,x_n,t,x)d\pmb{t},
\end{align}
where $\pmb{x}=(x_1,\ldots,x_n)$, $\pmb{t}=(t_1,\ldots,t_n)$, 
$T_n(t)=\{(t_1,\ldots,t_n)\in [0,t]^n;t_1<\ldots<t_n\}$, and
\begin{equation}
\label{def-phi-n}
\varphi_n(\pmb{x},x;t) = \int_{T_n(t)} \big(f_n(t_1,\cdot,\ldots,t_n,\cdot,t,x)*\k^{\otimes n}\big)(\pmb{x}) d\pmb{t}=\big(f_n(\cdot,x;t) * \k^{\otimes n}\big)(\pmb{x}).
\end{equation}

For any $n\geq 1$ and $(x_1,z_1), \ldots,(x_n,z_n) \in {\bf Z}$, we let
\begin{equation}
\label{def-fn*}
f_n^{*}(x_1,z_1,\ldots,x_n,z_n,x;t):=\prod_{i=1}^{n}z_i \, \varphi_n(x_1,\ldots,x_n,x;t),
\end{equation}
with the convention that $\varphi_n(\cdot,x;t)=f_n(\cdot,x;t)$ if Assumption A2 holds (i.e. $\k=\delta_0$).
\medskip

We have the following definition.

\begin{definition}
{\rm The process $\{u(t,x);t\geq 0,x\in \bR^d\}$ with $\bE|u(t,x)|^2<\infty$ for all $(t,x) \in \bR_{+} \times \bR^d$ and the Poisson-chaos expansion:
\begin{equation}
\label{series-u}
u(t,x)=1+\sum_{n\geq 1}I_n\big(f_n^{*}(\cdot,x;t)\big),
\end{equation}
is the {\bf (Skorohod) solution} of equation\eqref{PAM} (or \eqref{HAM}), provided that $f_n^{*}(\cdot,x;t)\in \cH^{\otimes n}$ for any $(t,x)\in \bR_{+}\times \bR^d$ and $n\geq 1$, and the series converges in $L^2(\Omega)$.
}
\end{definition}

We recall the following result; see Corollary 4.6 of \cite{BW26}.

\begin{theorem}
If Dalang's condition \eqref{dalang} holds, then equations \eqref{PAM} and \eqref{HAM} have unique solutions.
\end{theorem}

Recall the definition \eqref{def-FR} of the spatial integral $F_R(t)$.
Using \eqref{series-u}, we deduce that $F_R(t)$ has the Poisson-chaos expansion:
\[
F_R(t)=\sum_{n\geq 1} I_n\big(g_{n,R}(\cdot, t)\big), \quad \mbox{where} \quad g_{n,R}(\cdot,t)=\int_{B_R}f_n^{*}(\cdot,x;t)dx.
\]

Using definition \eqref{def-fn*} of $f_n^*(\cdot,x;t)$, we see that
\begin{equation}
\label{def-gnR}
g_{n,R}(x_1,z_1,\ldots,x_n,z_n,t)=\prod_{i=1}^{n} z_i \, Q_{n,R}(x_1,\ldots,x_n;t) 
\end{equation}
where
\begin{equation}
\label{def-QnR}
Q_{n,R}(x_1,\ldots,x_n;t) :=\int_{B_R}\varphi_n(x_1,\ldots,x_n,x;t)dx.
\end{equation}

\section{Limiting covariance}
\label{section-limit}

In this section, we give the proof of Theorem \ref{cov-th}. 

\medskip

Similarly to Lemma 4.1 of \cite{BZ24}, it can be proved that
for any $(t_1,x_1),\ldots,(t_k,x_k) \in  \bR_{+} \times \bR^d$ and $z \in \bR^d$,
\[
\big(u(t_1,x_1+z),\ldots,u(t_k,x_k+z)\big) \stackrel{d}{=} \big(u(t_1,x_1),\ldots,u(t_k,x_k)\big),
\]
where $\stackrel{d}{=}$ denotes equality in distribution. In particular, $\{u(t,x)\}_{x\in \bR^d}$ is strictly stationary, and
for any $t,s>0$,
\[
{\rm Cov}\big(u(t,x),u(s,y)\big)=:\rho_{t,s}(x-y) \quad \mbox{depends only on $x-y$}.
\]


\begin{lemma}
\label{lem-rho-int}
{Suppose that either one of conditions (i)-(iii) in Theorem \ref{cov-th} hold.} Then,
\begin{equation}
\label{rho-int}
\int_{\bR^d} \rho_{t,s}(x) dx<\infty \quad \mbox{for any $t,s>0$}.
\end{equation}
\end{lemma}

\begin{proof}
As in Remark 4.7 of \cite{BJ26}, we make a connection with the following Gaussian model:
\begin{align}
\label{PHAM-G}
	\begin{cases}
		\cL U(t,x)= \sqrt{m_2} \, U(t,x) \dot{W}(x), \
	t>0, \ x \in \bR^d \quad (d\leq 2), \\
		U(0,x) = 1, \ \dfrac{\partial U}{\partial t} (0,x)=0, \quad x \in \bR^d.
	\end{cases}
\end{align}
where $m_2$ is given by \eqref{def-m2}, and $\cL$ is the heat operator, or the wave operator. {If $k \in L^1(\bR^d)$,} $W$ is an isonormal Gaussian process $\{W(\varphi);\varphi \in \cP_0(\bR^d)\}$ with covariance
\[
\bE[W(\varphi)W(\psi)]=\langle \varphi,\psi \rangle_0,
\]
with inner product $\langle \cdot,\cdot \rangle_0$ given by \eqref{def-prod0}. {If $k=\delta_0$,} $W$ is a time-independent Gaussian white noise $\{W(\varphi);\varphi \in L^2(\bR^d)\}$. 

The solution of \eqref{PHAM-G} satisfies the integral equation:
\[
U(t,x)=1+\int_0^t \int_{\bR^d}G_{t-s}(x-y)U(s,y)W(\delta y) ds,
\]
where the integral is interpreted in the Skorohod sense.
It is known that the solution of \eqref{PHAM-G} exists and has the chaos expansion
\[
U(t,x)=1+\sum_{n\geq 1}m_2^{n/2} I_n^W \big( f_n(\cdot,x;t)\big),
\]
where $I_n^W$ is the multiple integral with respect to $W$, and $f_n(\cdot,x;t)$ is given by \eqref{def-fn-x}. 
See Theorem 3.2 of \cite{BY23} for the heat equation, respectively Theorem 2.1 of \cite{BY22} for the wave equation.

Note that $\widetilde{f}_n^{*}(x_1,z_1,\ldots,x_n,z_n,x;t)=\prod_{i=1}^{n}z_i (\widetilde{f}_n\big(\cdot,x;t)*\k^{\otimes}\big)(\pmb{x})$, and hence,
\[
 \langle \widetilde{f}_n^*(\cdot,x;t), \widetilde{f}_n^*(\cdot,0;s)\rangle_{\cH^{\otimes n}}= m_2^n
\langle \widetilde{f}_n(\cdot,x;t),  \widetilde{f}_n(\cdot,0;s)\rangle_{\cP_0(\bR^d)^{\otimes n}}.
\]
It follows that
\begin{align*}
\rho_{t,s}(x)&=\bE[(u(t,x)-1)(u(s,0)-1)]=\sum_{n\geq 1}n!\, \langle \widetilde{f}_n^*(\cdot,x;t), \widetilde{f}_n^*(\cdot,0;s)\rangle_{\cH^{\otimes n}}\\
&=\sum_{n\geq 1}n!\, m_2^n \langle \widetilde{f}_n(\cdot,x;t), \widetilde{f}_n(\cdot,0;s)\rangle_{\cP_0(\bR^d)^{\otimes n}}=\bE[(U(t,x)-1)(U(s,0)-1)]=:\rho_{t,s}^U(x).
\end{align*}
The conclusion follows, since 
\[
\int_{\bR^d}\rho_{t,s}^{U}(x)dx<\infty.
\]
See Theorem 1.1 of \cite{BY23} for the heat equation, and Theorem 1.3 of \cite{BY22} for the wave equation.
\end{proof}

\begin{lemma}
\label{lem-lim-cov}
{Suppose that either one of conditions (i)-(iii) in Theorem \ref{cov-th} hold.}
Then,
\begin{equation}
\label{def-cts}
\lim_{R \to \infty}\frac{1}{R^{d}}\bE[F_R(t)F_R(s)] =\omega_d\int_{\bR^d}\rho_{t,s}(x)dx=:c(t,s), \quad \mbox{for any $t,s>0$},
\end{equation}
where $\omega_d$ is the volume of the unit ball $B_1$ in $\bR^d$, i.e. $\omega_1=2$ and $\omega_2=\pi$. 
\end{lemma}

\begin{proof}
By the dominated convergence theorem,
\begin{align*}
\frac{1}{R^{d}}\bE[F_R(t)F_R(s)] & =\frac{1}{R^{d}}\int_{B_R} \int_{B_R} \rho_{t,s}(x-y)dxdy=\frac{1}{R^{d}} \int_{B_{2R}} \rho_{t,s}(z) {\rm Vol} \big(B_R \cap B_R(-z)\big)dz \\
& \longrightarrow \omega_{d} \int_{\bR^d} \rho_{t,s}(z) dz, \quad \mbox{as} \ R \to \infty.
\end{align*}
The application of this theorem is justified by Lemma \ref{lem-rho-int}. 
\end{proof}

\section{Proof of Theorem \ref{CLT}}
\label{section-CLT}

In this section, we give the proof of Theorem \ref{CLT}.

\medskip

{
First, we study the asymptotic behaviour of the variance of the projection of $F_R(t)$ on a fixed chaos space of order $n$. More precisely, we denote:}
\[
\sigma_{n,R}^2(t):={\rm Var}\big(I_n\big(g_{n,R}(\cdot, t)\big)\big).
\]
{By orthogonality, $\sigma_R^2(t)=\sum_{n\geq 1}\sigma_{n,R}^2(t)$.} Similarly to Lemma \ref{lem-lim-cov},
\begin{align}
\label{sigma-nR1}
\frac{1}{R^d}\sigma_{n,R}^2(t)& =\frac{1}{R^d} \int_{B_R} \int_{B_R} \rho_{t}^{(n)}(x-y)dxdy =\frac{1}{R^{d}} \int_{B_{2R}} \rho_{t}^{(n)}(z) {\rm Vol} \big(B_R \cap B_R(-z)\big)dz \\
\label{sigma-nR}
& \longrightarrow \omega_{d} \int_{\bR^d} \rho_{t}^{(n)}(z) dz=:c_n(t), \quad \mbox{as} \ R \to \infty,
\end{align}
where
\[
\rho_{t}^{(n)}(x)=\bE\big[I_n(f_n^*(\cdot, x;t)) I_n(f_n^*(\cdot, 0;t)) \big]=n!\, \langle \widetilde{f}_n^*(\cdot,x;t), \widetilde{f}_n^*(\cdot,0;t)\rangle_{\cH^{\otimes n}}.
\]

{Therefore, $c(t)=\sum_{n\geq 1}c_n(t)$, where $c(t)=c(t,t)$ is given by \eqref{def-cts}. }

\medskip

We will use the following Fourth Moment Theorems on the Poisson Space.

\begin{theorem}[Corollary 1.3 of \cite{DVZ18}]
\label{4th-theorem}
For each $n\geq 1$, let $F_n=I_{q_n}(f_n)$ for some  $f_n \in \cH^{\otimes n}$ and $q_n\geq 1$. If $\bE(F_n^2) \to 1$ and $\bE(F_n^4) \to 3$, then 
\[
F_n \stackrel{d}{\longrightarrow} Z \sim \cN(0,1).
\]
\end{theorem}

\begin{theorem}[Corollary 1.8 of \cite{DVZ18}]
\label{4th-theorem-m}
For each $n\geq 1$, let $F_n=(F_n^{(1)},\ldots,F_n^{(d)})$ be such that $F_n^{(i)}=I_{q_k}(f_n^{(i)})$ for some  $f_n^{(i)} \in \cH^{\otimes n}$ and $q_i\geq 1$, for all $i=1,\ldots,d$. Let $\pmb{\Sigma}=(\Sigma_{i,j})_{1\leq i,j\leq d}$ be a symmetric non-negative definite matrix.  If the following conditions hold:
\begin{description}
\item[(a)] ${\rm Cov} (F_n^{(i)},F_n^{(j)}) \to \Sigma_{i,j}$ as $n \to \infty$, for all $i,j=1,\ldots,d$,
\item[(b)] $\bE[(F_n^{(i)})^4] \to 3 \Sigma_{i,i}^2$ for all $i=1,\ldots,d$,
\end{description}
then 
\[
F_n\stackrel{d}{\longrightarrow} Z\sim \cN(\pmb{0}, \pmb{\Sigma}),
\]
where $\cN(\pmb{0},\pmb{\Sigma})$ is the multivariate normal distribution with mean $\pmb{0}$ and variance $\pmb{\Sigma}$.

\end{theorem}

\begin{lemma}
\label{lem-abc}
{
Suppose that either one of conditions (i)-(iii) of Theorem \ref{cov-th} hold. In case (iii), we assume in addition that $\k \in L^2(\bR^2)$.} For any $t>0$ and $n\geq 1$, 
\begin{equation}
\label{CLT1}
\frac{1}{\sigma_{n,R}(t)}I_n\big(g_{n,R}(\cdot, t)\big) \stackrel{d}{\to} N(0,1) \quad \mbox{as $R \to \infty$},
\end{equation}
or equivalently,
\begin{equation}
\label{CLT2}
\frac{1}{R^{d/2}}I_n\big(g_{n,R}(\cdot, t)\big) \stackrel{d}{\to} N(0,c_n(t)) \quad \mbox{as $R \to \infty$}.
\end{equation}
\end{lemma}

\begin{proof}
It is enough to prove \eqref{CLT1}. Then \eqref{CLT2} will follow, since $\sigma_{n,R}^2(t) \sim c_n(t) R^d$ by \eqref{sigma-nR}.

We drop $t$ from the notation. 
By Theorem \ref{4th-theorem}, we have to prove that
\begin{equation}
\label{4mom-I1}
\lim_{R \to \infty}\frac{\bE\big[I_n^4(g_{n,R})]}{\sigma_{n,R}^4} = 3,
\end{equation}
that is
\begin{equation}
\label{4mom-I2}
\lim_{R \to \infty}\frac{\bE\big[I_n^4(g_{n,R})]-3\sigma_{n,R}^4}{\sigma_{n,R}^4}=0.
\end{equation}

We use the fact that $I_n(g_{n,R})=I_n(\widetilde{g}_{n,R})$.
By Lemma \ref{cumulant}, $\bE\big[I_n^4(\widetilde{g}_{n,R})]-3\sigma_{n,R}^4$ is bounded above by a linear combination of terms of the form
\[
\| \widetilde{g}_{n,R} \star_r^{\ell} \widetilde{g}_{n,R} \|_{\cH^{\otimes(2n-r-\ell)}}^2 \quad \mbox{with $0\leq \ell \leq r \leq n$ and $(\ell,r) \not=(n,n)$}.
\]

Relation \eqref{4mom-I2} follows by Theorem \ref{gg-lem} below, using the fact that $\sigma_{n,R}^4 \sim c_n^2 R^{2d}$.
\end{proof}

\begin{theorem}
\label{gg-lem}
{Suppose that either one of conditions (i)-(iii) of Theorem \ref{cov-th} hold. In case (iii), we assume in addition that $\k \in L^{\infty}(\bR^2)$.}

Let $g_{n,R}=g_{n,R}(\cdot, t)$ for fixed $t>0$ and $n\geq 1$.
For any $0\leq \ell \leq r \leq n$ with $(\ell,r) \not=(n,n)$ and $(\ell,r) \not=(0,0)$, we have:
\[
\lim_{R\to \infty}\frac{1}{R^{2d}}\| \widetilde{g}_{n,R} \star_r^{\ell} \widetilde{g}_{n,R} \|_{\cH^{\otimes(2n-r-\ell)}}^2 =0.
\]
\end{theorem}

\begin{proof}
We consider the variables 
\begin{align*}
\pmb{\gamma}_{r-\ell}&=(\gamma_1,\ldots,\gamma_{r-\ell}) \quad \mbox{with} \quad \gamma_i=(u_i,\zeta_i) \in \bR^d \times \bR_0, \ \mbox{for} \ i=1,\ldots,r-\ell \\
\pmb{y}_{n-r}  &=(y_1,\ldots,y_{n-r}) \quad \mbox{with} \quad y_i=(v_i,a_i) \in \bR^d \times \bR_0,  \ \mbox{for} \ i=1,\ldots,n-r \\
\pmb{y}_{n-r}' &=(y_1',\ldots,y_{n-r}') \quad \mbox{with} \quad y_i'=(v_i',a_i') \in \bR^d \times \bR_0,  \ \mbox{for} \ i=1,\ldots,n-r \\
\pmb{\xi}_{\ell}&=(\xi_1,\ldots,\xi_{\ell}) \quad \quad \ \mbox{with} \quad \xi_i=(x_i,z_i) \in \bR^d \times \bR_0,  \ \mbox{for} \ i=1,\ldots,\ell.
\end{align*}

By definition of  $\star_r^{\ell}$ (see Appendix \ref{appA}) and relation \eqref{def-gnR},
\begin{align*}
\big(\widetilde{g}_{n,R} \star_r^{\ell} \widetilde{g}_{n,R} \big)(\pmb{\gamma}_{r-\ell},\pmb{y}_{n-r},\pmb{y}_{n-r}' ) & =\int_{Z^{\ell}} 
\widetilde{g}_{n,R} (\pmb{\xi}_{\ell},\pmb{\gamma}_{r-\ell},\pmb{y}_{n-r})\widetilde{g}_{n,R} (\pmb{\xi}_{\ell},\pmb{\gamma}_{r-\ell},\pmb{y}_{n-r}') \fm^{\otimes \ell}(d\pmb{\xi}_{\ell}) \\
&= m_2^{\ell} \prod_{i=1}^{r-\ell} \zeta_i^2 \prod_{i=1}^{n-r}a_i \prod_{i=1}^{n-r}a_i'
\big(\widetilde{Q}_{n,R} \star_r^{\ell} \, \widetilde{Q}_{n,R} \big)(\pmb{u}_{r-\ell},\pmb{v}_{n-r},\pmb{v}_{n-r}' \big),
\end{align*}
where $\pmb{u}_{r-\ell}=(u_1,\ldots,u_{r-\ell})$, $\pmb{v}_{n-r}=(v_1,\ldots,v_{n-r})$ and $\pmb{v}_{n-r}'=(v_1,\ldots,v_{n-r}')$. Hence,
\begin{align}
\nonumber
\big\|\widetilde{g}_{n,R} \star_r^{\ell} \widetilde{g}_{n,R} \big\|_{\cH^{\otimes (2n-r-\ell)}}^2 & =\int_{Z^{2n-r-\ell}} \big[\big(\widetilde{g}_{n,R} \star_r^{\ell} \widetilde{g}_{n,R} \big)(\pmb{\gamma}_{r-\ell},\pmb{y}_{n-r},\pmb{y}_{n-r}' ) \big]^2 \fm^{\otimes(2n-r-\ell)}(d\pmb{\gamma}_{r-\ell},d\pmb{y}_{n-r},d\pmb{y}_{n-r}' )\\
\label{gg-QQ}
& =m_2^{2\ell+2(n-r)} m_4^{r-\ell} \, \|\widetilde{Q}_{n,R} \star_r^{\ell} \widetilde{Q}_{n,R} \|_{L^2(\bR^{d(2n-r-\ell)})}^2.
\end{align}

We evaluate separately the last norm above. We let $\pmb{x}_{\ell}=(x_1,\ldots,x_{\ell})$. Then,
\begin{align}
\nonumber
& \|\widetilde{Q}_{n,R} \star_r^{\ell} \widetilde{Q}_{n,R} \|_{L^2(\bR^{d(2n-r-\ell)})}^2  \\
\nonumber
& \quad =\int_{(\bR^d)^{2n-r-\ell}} \left[\int_{\bR^{d\ell}}  \widetilde{Q}_{n,R} (\pmb{x}_{\ell},\pmb{u}_{r-\ell},\pmb{v}_{n-r}) \widetilde{Q}_{n,R}(\pmb{x}_{\ell},\pmb{u}_{r-\ell},\pmb{v}'_{n-r}) d\pmb{x}_{\ell} \right]^2 d\pmb{u}_{r-\ell} d\pmb{v}_{n-r} d\pmb{v}'_{n-r}\\
\nonumber
& \quad =\int_{(\bR^d)^{2n-r-\ell}} \int_{(\bR^d)^{\ell}} \int_{(\bR^d)^{\ell}}  \widetilde{Q}_{n,R} (\pmb{x}_{\ell},\pmb{u}_{r-\ell},\pmb{v}_{n-r}) \widetilde{Q}_{n,R}(\pmb{x}_{\ell},\pmb{u}_{r-\ell},\pmb{v}'_{n-r}) 
\\
\label{norm-Q}
 & \qquad  \qquad \qquad \qquad  \widetilde{Q}_{n,R} (\pmb{x}_{\ell}',\pmb{u}_{r-\ell},\pmb{v}_{n-r}) \widetilde{Q}_{n,R}(\pmb{x}_{\ell}',\pmb{u}_{r-\ell},\pmb{v}'_{n-r})
 d\pmb{x}_{\ell} d\pmb{x}'_{\ell} d\pmb{u}_{r-\ell} d\pmb{v}_{n-r} d\pmb{v}'_{n-r}.
\end{align}

Since
\[
\widetilde{Q}_{n,R}(x_1,\ldots,x_n;t)=\int_{B_R} \widetilde{\varphi}_{n}(x_1,\ldots,x_n,x;t)dx,
\]
we have:
\begin{align*}
& \|\widetilde{Q}_{n,R} \star_r^{\ell} \widetilde{Q}_{n,R} \|_{L^2(\bR^{d(2n-r-\ell)})}^2  =\int_{(\bR^d)^{2n-r+\ell}} \int_{B_R^4} \widetilde{\varphi}_{n} (\pmb{x}_{\ell},\pmb{u}_{r-\ell},\pmb{v}_{n-r},a_1;t) \widetilde{\varphi}_{n}(\pmb{x}_{\ell},\pmb{u}_{r-\ell},\pmb{v}'_{n-r},a_2;t) \\
& \quad \quad \quad \quad \widetilde{\varphi}_{n} (\pmb{x}_{\ell}',\pmb{u}_{r-\ell},\pmb{v}_{n-r},a_3;t) \widetilde{\varphi}_{n}(\pmb{x}_{\ell}',\pmb{u}_{r-\ell},\pmb{v}'_{n-r},a_4;t)
d\pmb{a}_4 \, d\pmb{x}_{\ell} \, d\pmb{x}'_{\ell} \, d\pmb{u}_{r-\ell} \, d\pmb{v}_{n-r} \, d\pmb{v}'_{n-r},
\end{align*}
where $\pmb{a}_4=(a_1,a_2,a_3,a_4)$. Note that for any $a \in \bR^d$,
\begin{equation}
\label{a-minus}
\widetilde{\varphi}_n(x_1,\ldots,x_n,x;t)=
\widetilde{\varphi}_n(x_1-a,\ldots,x_n-a,x-a;t).
\end{equation}

We use \eqref{a-minus} with $a=a_1$, followed by the change of variables 
\[
(\overline{\pmb{x}}_{\ell},\overline{\pmb{x}'}_{\ell}, \overline{\pmb{u}}_{r-\ell},\overline{\pmb{v}}_{n-r},\overline{\pmb{v}'}_{n-r})=
(\pmb{x}_{\ell}-a_1,\pmb{x}'_{\ell}-a_1,\pmb{u}_{r-\ell}-a_1,\pmb{v}_{n-r}-a_1,\pmb{v}'_{n-r}-a_1),
\]
where $\pmb{x}_{\ell}-a_1=(x_1-a_1,\ldots,x_{\ell}-a_1)$ and a similar notation for the other variables. Hence,
\begin{align*}
& \|\widetilde{Q}_{n,R} \star_r^{\ell} \widetilde{Q}_{n,R} \|_{L^2(\bR^{d(2n-r-\ell)})}^2  =\int_{(\bR^d)^{2n-r+\ell}} \int_{B_R^4} \widetilde{\varphi}_{n} (\overline{\pmb{x}}_{\ell},\overline{\pmb{u}}_{r-\ell},\overline{\pmb{v}}_{n-r},0;t) \widetilde{\varphi}_{n}(\overline{\pmb{x}}_{\ell},\overline{\pmb{u}}_{r-\ell},\overline{\pmb{v}}'_{n-r},a_2-a_1;t) \\
& \quad \quad \quad  \widetilde{\varphi}_{n} (\overline{\pmb{x}}_{\ell}',\overline{\pmb{u}}_{r-\ell},\overline{\pmb{v}}_{n-r},a_3-a_1;t) \widetilde{\varphi}_{n}(\overline{\pmb{x}}_{\ell}',\overline{\pmb{u}}_{r-\ell},\overline{\pmb{v}}'_{n-r},a_4-a_1;t)
d\pmb{a}_4 \, d\overline{\pmb{x}}_{\ell} \, d\overline{\pmb{x}}'_{\ell} \, d\overline{\pmb{u}}_{r-\ell} \, d\overline{\pmb{v}}_{n-r} \, d\overline{\pmb{v}}'_{n-r}.
\end{align*}
Next, we use the change of variables $b_1=a_2-a_1$, $b_2=a_3-a_1$, $b_3=a_4-a_1$. We denote by $B_R(a)=\{x \in \bR^d;|x-a|\leq R\}$ the ball of center $a \in \bR^d$ and radius $R$. We obtain:
\begin{align*}
& \|\widetilde{Q}_{n,R} \star_r^{\ell} \widetilde{Q}_{n,R} \|_{L^2(\bR^{d(2n-r-\ell)})}^2  =\int_{(\bR^d)^{2n-r+\ell}} \int_{B_R} \int_{(B_R(-a_1))^3} \widetilde{\varphi}_{n} (\overline{\pmb{x}}_{\ell},\overline{\pmb{u}}_{r-\ell},\overline{\pmb{v}}_{n-r},0;t) \widetilde{\varphi}_{n}(\overline{\pmb{x}}_{\ell},\overline{\pmb{u}}_{r-\ell},\overline{\pmb{v}}'_{n-r},b_1;t) \\
& \quad \quad \quad  \widetilde{\varphi}_{n} (\overline{\pmb{x}}_{\ell}',\overline{\pmb{u}}_{r-\ell},\overline{\pmb{v}}_{n-r},b_2;t) \widetilde{\varphi}_{n}(\overline{\pmb{x}}_{\ell}',\overline{\pmb{u}}_{r-\ell},\overline{\pmb{v}}'_{n-r},b_3;t)
d\pmb{b}_3 da_1 \, d\overline{\pmb{x}}_{\ell} \, d\overline{\pmb{x}}'_{\ell} \, d\overline{\pmb{u}}_{r-\ell} \, d\overline{\pmb{v}}_{n-r} \, d\overline{\pmb{v}}'_{n-r},
\end{align*}
where $\pmb{b}_3=(b_1,b_2,b_3)$.
We bound the integral above, by passing from $(B_R(-a_1))^3$ to $(\bR^d)^3$ so that the integrand becomes independent of $a_1$. Recall that the volume of $B_R$ is
\[
{\rm Vol}(B_R)=\omega_d R^d \quad \mbox{where} \quad \omega_d=\frac{2\pi^{d/2}}{d \Gamma(d/2)}.
\]
This leads to:
\begin{equation}
\label{QQ-bound}
\|\widetilde{Q}_{n,R} \star_r^{\ell} \widetilde{Q}_{n,R} \|_{L^2(\bR^{d(2n-r-\ell)})}^2  \leq \omega_d R^d J,
\end{equation}
where
\begin{align*}
J &:= \int_{(\bR^d)^{2n-r+\ell+3}}  \widetilde{\varphi}_{n} (\pmb{x}_{\ell},\pmb{u}_{r-\ell},\pmb{v}_{n-r},0;t) \widetilde{\varphi}_{n}(\pmb{x}_{\ell},\pmb{u}_{r-\ell},\pmb{v}'_{n-r},b_1;t) \\
& \quad \quad \quad  \widetilde{\varphi}_{n} (\pmb{x}_{\ell}',\pmb{u}_{r-\ell},\pmb{v}_{n-r},b_2;t) \widetilde{\varphi}_{n}(\pmb{x}_{\ell}',\pmb{u}_{r-\ell},\pmb{v}'_{n-r},b_3;t)
d\pmb{x}_{\ell} \, d\pmb{x}'_{\ell} \, d\pmb{u}_{r-\ell} \, d\pmb{v}_{n-r} \, d\pmb{v}'_{n-r}  d\pmb{b}_3.
\end{align*}

By \eqref{gg-QQ} and \eqref{QQ-bound}, it follows that:
\[
\frac{1}{R^{2d}}\big\|\widetilde{g}_{n,R} \star_r^{\ell} \widetilde{g}_{n,R} \big\|_{\cH^{\otimes (2n-r-\ell)}}^2 \leq m_2^{2\ell+2(n-r)} m_4^{r-\ell} \omega_d \frac{1}{R^d} J.
\]

Therefore, the proof of the theorem will be complete, once we show that
\begin{equation}
\label{J-finite}
\fbox{$J<\infty$}.
\end{equation}

\underline{We now start the proof of \eqref{J-finite}.} This will require a major effort. { We will use graph theory. The graph will be defined differently for the white noise, and the colored noise.}

First, by definition,
$$
\widetilde{\varphi}_n(,\cdot,x;t)=\frac{1}{n!}\sum_{\rho \in S_n}\varphi_{n,\rho}(\cdot,x;t),$$
where $\varphi_{n,\rho}(x_1,\ldots,x_n,,x;t)=\varphi_{n}(x_{\rho(1)},\ldots, x_{\rho(n)},x;t)$. Hence,
\begin{align*}
J & = \frac{1}{(n!)^4} \sum_{\rho_1,\ldots,\rho_4 \in S_n} J_{\rho_1,\ldots,\rho_4},
\end{align*}
where 
\begin{align}
\nonumber
& J_{\rho_1,\ldots,\rho_4} =\int_{(\bR^d)^{2n-r+\ell+3}}  \varphi_{n,\rho_1} (\pmb{x}_{\ell},\pmb{u}_{r-\ell},\pmb{v}_{n-r},0;t) \varphi_{n,\rho_2}(\pmb{x}_{\ell},\pmb{u}_{r-\ell},\pmb{v}'_{n-r},b_1;t) \\
\label{J-rho-def}
& \quad  \varphi_{n,\rho_3} (\pmb{x}_{\ell}',\pmb{u}_{r-\ell},\pmb{v}_{n-r},b_2;t) \varphi_{n,\rho_4}(\pmb{x}_{\ell}',\pmb{u}_{r-\ell},\pmb{v}'_{n-r},b_3;t)
d\pmb{x}_{\ell} \, d\pmb{x}'_{\ell} \, d\pmb{u}_{r-\ell} \, d\pmb{v}_{n-r} \, d\pmb{v}'_{n-r} d\pmb{b}_3.
\end{align}
Therefore, to prove \eqref{J-finite}, it is enough to show that for any $\rho_1,\ldots,\rho_4 \in S_n$ fixed,
\begin{equation}
\label{J-rho}
J_{\rho_1,\ldots\rho_4}<\infty.
\end{equation}

We consider separately the white noise and the colored noise.

\bigskip

{\bf Case 1. (white noise, $d=1$)} {In this case, $\varphi_n(\cdot,x;t)=f_n(\cdot,x;t)$ and
\begin{align*}
& J_{\rho_1,\ldots,\rho_4} =\int_{(\bR^d)^{2n-r+\ell+3}}  f_{n,\rho_1} (\pmb{x}_{\ell},\pmb{u}_{r-\ell},\pmb{v}_{n-r},0;t) f_{n,\rho_2}(\pmb{x}_{\ell},\pmb{u}_{r-\ell},\pmb{v}'_{n-r},b_1;t) \\
& \quad  f_{n,\rho_3} (\pmb{x}_{\ell}',\pmb{u}_{r-\ell},\pmb{v}_{n-r},b_2;t) f_{n,\rho_4}(\pmb{x}_{\ell}',\pmb{u}_{r-\ell},\pmb{v}'_{n-r},b_3;t)
d\pmb{x}_{\ell} \, d\pmb{x}'_{\ell} \, d\pmb{u}_{r-\ell} \, d\pmb{v}_{n-r} \, d\pmb{v}'_{n-r} d\pmb{b}_3,
\end{align*}
where $f_{n,\rho}(x_1,\ldots,x_n,,x;t)=f_{n}(x_{\rho(1)},\ldots, x_{\rho(n)},x;t)$.
We write
\begin{align*}
f_n(x_1,\ldots,x_n,x;t)&=\int_{T_n(t)} G_{t-t_n}(x-x_n) \ldots G_{t_2-t_1}(x_2-x_1) d\pmb{t}\\
&=\int_{D_n(t)} h_n(x_1,\ldots,x_n,\pmb{\delta},x;t) d\pmb{\delta},
\end{align*}
}
where $\pmb{\delta}=(\delta_1,\ldots,\delta_n)$, $D_n(t)=\{\pmb{\delta} \in [0,t]^n;\delta_1+\ldots+\delta_n<t\}$, and we denote:
\[
h_n(x_1,\ldots,x_n,\pmb{\delta},x;t):=G_{\delta_n}(x-x_n) \ldots G_{\delta_2}(x_2-x_1).
\]

Similarly, for any permutation $\rho \in S_n$,
\[
f_{n,\rho}(x_1,\ldots,x_n,x;t)=\int_{D_n(t)} h_{n,\rho}(x_1,\ldots,x_n,\pmb{\delta},x;t) d\pmb{\delta}
\]
where $h_{n,\rho}(x_1,\ldots,x_n,\pmb{\delta},x;t)=h_{n}(x_{\rho(1)},\ldots, x_{\rho(n)},\pmb{\delta},x;t)$. It follows that
\begin{equation}
\label{J-rho14}
J_{\rho_1,\ldots\rho_4}=\int_{D_n^4(t)} J_{\rho_1,\ldots,\rho_4}(\pmb{\delta}^{(1)},\ldots,\pmb{\delta}^{(4)})
d\pmb{\delta}^{(1)} \ldots d\pmb{\delta}^{(4)},
\end{equation}
where
\begin{align*}
& J_{\rho_1,\ldots,\rho_4}(\pmb{\delta}^{(1)},\ldots, \pmb{\delta}^{(4)}) =\int_{(\bR^d)^{2n-r+\ell+3}}  h_{n,\rho_1} (\pmb{x}_{\ell},\pmb{u}_{r-\ell},\pmb{v}_{n-r},\pmb{\delta}^{(1)},0;t)  \\
& \quad \quad \quad \quad \quad h_{n,\rho_2}(\pmb{x}_{\ell},\pmb{u}_{r-\ell},\pmb{v}'_{n-r},\pmb{\delta}^{(2)},b_1;t)   h_{n,\rho_3}(\pmb{x}'_{\ell},\pmb{u}_{r-\ell},\pmb{v}_{n-r},\pmb{\delta}^{(3)},b_2;t)  \\
& \quad \quad \quad \quad \quad h_{n,\rho_4}(\pmb{x}'_{\ell},\pmb{u}_{r-\ell},\pmb{v}'_{n-r},\pmb{\delta}^{(4)},b_3;t)  \, \, d\pmb{x}_{\ell} \, d\pmb{x}'_{\ell} \, d\pmb{u}_{r-\ell} \, d\pmb{v}_{n-r} \, d\pmb{v}'_{n-r} d\pmb{b}_3.
\end{align*}

\underline{Consider first the case when all permutations $\rho_i$'s are the identity.} In this case, $h_{n,\rho_i}=h_n$ for all $i=1,\ldots,4$. We call the corresponding integral $J_0(\pmb{\delta}^{(1)},\ldots, \pmb{\delta}^{(4)})$, and we let
\[
J_0:=\int_{D_n^4(t)} J_0(\pmb{\delta}^{(1)},\ldots,\pmb{\delta}^{(4)})d\pmb{\delta}^{(1)} \ldots d\pmb{\delta}^{(4)}.
\]

We will prove below that \fbox{$J_0<\infty$}, i.e. \eqref{J-rho} holds when all $\rho_i$'s are the identity.

\medskip
To see what $J_0(\pmb{\delta}^{(1)},\ldots,\pmb{\delta}^{(4)})$ looks like, we consider the example $n=7$, $\ell=2$ and $r=5$. Then
\begin{align*}
& J_{0}(\pmb{\delta}^{(1)},\ldots, \pmb{\delta}^{(4)}) = \int_{\bR^{14}}  d\pmb{x}_{2} \, d\pmb{x}'_{2} \, d\pmb{u}_{3} \, d\pmb{v}_{2} \, d\pmb{v}'_{2} d\pmb{b}_3 \\
&  G_{\delta_7^{(1)}}(0-v_2)
G_{\delta_6^{(1)}}(v_2-v_1) G_{\delta_5^{(1)}}(v_1-u_3)
G_{\delta_4^{(1)}}(u_3-u_2)G_{\delta_3^{(1)}}(u_2-u_1)
G_{\delta_2^{(1)}}(u_1-x_2) G_{\delta_1^{(1)}}(x_2-x_1)  \\
& G_{\delta_7^{(2)}}(b_1-v_2')
G_{\delta_6^{(2)}}(v_2'-v_1') G_{\delta_5^{(2)}}(v_1'-u_3)
G_{\delta_4^{(2)}}(u_3-u_2)G_{\delta_3^{(2)}}(u_2-u_1)
G_{\delta_2^{(2)}}(u_1-x_2) G_{\delta_1^{(2)}}(x_2-x_1) \\
& G_{\delta_7^{(3)}}(b_2-v_2)
G_{\delta_6^{(3)}}(v_2-v_1) G_{\delta_5^{(3)}}(v_1-u_3)
G_{\delta_4^{(3)}}(u_3-u_2)G_{\delta_3^{(3)}}(u_2-u_1)
G_{\delta_2^{(3)}}(u_1-x_2') G_{\delta_1^{(3)}}(x_2'-x_1') \\
& G_{\delta_7^{(4)}}(b_3-v_2')
G_{\delta_6^{(4)}}(v_2'-v_1') G_{\delta_5^{(4)}}(v_1'-u_3)
G_{\delta_4^{(4)}}(u_3-u_2)G_{\delta_3^{(4)}}(u_2-u_1)
G_{\delta_2^{(4)}}(u_1-x_2') G_{\delta_1^{(4)}}(x_2'-x_1').
\end{align*}

To represent the product of the $4n$ factors $G_{\delta_i^{(j)}}$ with $i=1,\ldots,n$ and $j=1,\ldots,4$, we use a multigraph {$G_{\rm white}=(V_{
\rm white},E_{\rm white})$} with the set of vertices:
\[
V_{\rm white}=\{x_1,\ldots,x_{\ell},x_1',\ldots,x_{\ell}',u_1,\ldots, u_{r-\ell},v_1,\ldots,v_{n-r},v_1',\ldots,v_{n-r}',0,b_1,b_2,b_3\}.
\]
We connect two vertices $u$ and $v$ by an edge if the product $G_{\delta_i^{(j)}}$ appears in the product above for some $i=1,\ldots,n$ and $j=1,\ldots,4$. Figure \ref{pic-J0} shows this graph for $n=7$, $\ell=2$ and $r=5$.

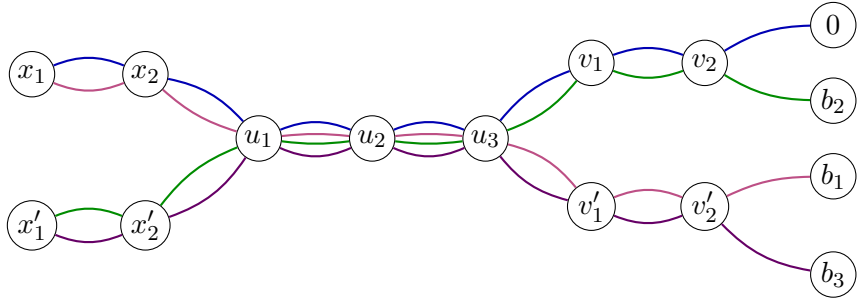
\begin{figure}[htbp]
    \centering
\begin{tikzpicture}[
    vertex/.style={
        circle,
        draw=black,
        fill=white,
        minimum size=6mm,
        inner sep=1pt,
        font=\small
    },
    blueedge/.style={draw=blue!70!black, thick},
    rededge/.style={draw=magenta!75!black, thick},
    greenedge/.style={draw=green!55!black, thick},
    purpleedge/.style={draw=violet!80!black, thick}
]


\node[vertex] (x1) at (0,2) {$x_1$};
\node[vertex] (x2) at (1.5,2) {$x_2$};

\node[vertex] (x1p) at (0,0) {$x_1'$};
\node[vertex] (x2p) at (1.5,0) {$x_2'$};

\node[vertex] (u1) at (3,1.15) {$u_1$};
\node[vertex] (u2) at (4.5,1.15) {$u_2$};
\node[vertex] (u3) at (6,1.15) {$u_3$};

\node[vertex] (v1) at (7.4,2.15) {$v_1$};
\node[vertex] (v2) at (8.9,2.15) {$v_2$};

\node[vertex] (zero) at (10.6,2.65) {$0$};
\node[vertex] (b2)   at (10.6,1.65) {$b_2$};

\node[vertex] (v1p) at (7.4,0.25) {$v_1'$};
\node[vertex] (v2p) at (8.9,0.25) {$v_2'$};

\node[vertex] (b1) at (10.6,0.65) {$b_1$};
\node[vertex] (b3) at (10.6,-0.65) {$b_3$};


\draw[blueedge]
    (x1) to[bend left=22] (x2);

\draw[rededge]
    (x1) to[bend right=22] (x2);

\draw[blueedge]
    (x2) to[bend left=18] (u1);

\draw[rededge]
    (x2) to[bend right=12] (u1);


\draw[greenedge]
    (x1p) to[bend left=22] (x2p);

\draw[purpleedge]
    (x1p) to[bend right=22] (x2p);

\draw[greenedge]
    (x2p) to[bend left=15] (u1);

\draw[purpleedge]
    (x2p) to[bend right=18] (u1);


\draw[blueedge]
    (u1) to[bend left=22] (u2);

\draw[rededge]
    (u1) to[bend left=6] (u2);

\draw[greenedge]
    (u1) to[bend right=7] (u2);

\draw[purpleedge]
    (u1) to[bend right=24] (u2);


\draw[blueedge]
    (u2) to[bend left=22] (u3);

\draw[rededge]
    (u2) to[bend left=6] (u3);

\draw[greenedge]
    (u2) to[bend right=7] (u3);

\draw[purpleedge]
    (u2) to[bend right=24] (u3);


\draw[blueedge]
    (u3) to[bend left=15] (v1);

\draw[greenedge]
    (u3) to[bend right=15] (v1);

\draw[blueedge]
    (v1) to[bend left=20] (v2);

\draw[greenedge]
    (v1) to[bend right=20] (v2);

\draw[blueedge]
    (v2) to[bend left=15] (zero);

\draw[greenedge]
    (v2) to[bend right=15] (b2);


\draw[rededge]
    (u3) to[bend left=18] (v1p);

\draw[purpleedge]
    (u3) to[bend right=18] (v1p);

\draw[rededge]
    (v1p) to[bend left=22] (v2p);

\draw[purpleedge]
    (v1p) to[bend right=22] (v2p);

\draw[rededge]
    (v2p) to[bend left=12] (b1);

\draw[purpleedge]
    (v2p) to[bend right=18] (b3);

\end{tikzpicture}

\caption{Graph for representing $J_0$ with $n=7$, $\ell=2$ and $r=5$.}
    \label{pic-J0}
\end{figure}

We apply Lemma \ref{lemB} with $g_t=G_t$, where $G_t$ is the fundamental solution of the heat equation, respectively of the wave equation. For both equations, $G_t \in L^{\infty}(\bR)\cap L^{1}(\bR)$:
\begin{equation}
\label{bound-G}
\|G_t\|_{L^{\infty}(\bR)} =\left\{
\begin{array}{ll} 
(2\pi t)^{-1/2}& \mbox{for heat equation} \\
1/2 & \mbox{for wave equation}
\end{array} \right.
\quad 
\|G_t\|_{L^{1}(\bR)} =\left\{
\begin{array}{ll} 
1 & \mbox{for heat equation} \\
t & \mbox{for wave equation}
\end{array} \right.
\end{equation}
The fixed vertex (for which there is no integration variable in $J_0$) is $p=0$. For any spanning tree $T$ of {$G_{\rm white}$},
\begin{equation}
\label{J0-bound}
J_0 (\pmb{\delta}^{(1)},\ldots,\pmb{\delta}^{(4)}) \leq \prod_{i=1}^4 \left( \prod_{j \in J_i} \|G_{\delta_j^{(i)}}\|_{L^{\infty}(\bR)} \prod_{j \in I_i} \|G_{\delta_j^{(i)}}\|_{L^{1}(\bR)}\right),
\end{equation}
where $I_i$ contains the indices $j\leq n$ which correspond to tree edges, and $J_i=\{1,\ldots,n\}\verb2\2 I_i$.

For instance, in the example with $n=7$, $\ell=2$ and $r=5$, we consider the spanning tree $T$ with the set $E_T$ of edges selected as follows, from the 4 colored paths in Figure \ref{pic-J0}:
\begin{align*}
\mbox{from the blue path:}  & \ \{x_1,x_2\}, \{x_2,u_1\},\{u_1,u_2\},\{v_1,v_2\},\{v_2,0\} \\
\mbox{from the pink path:} & \ \{v_1',v_2'\},\{v_2',b_1\} \\
\mbox{from the green path:}  & \ \{x_2',u_1\},\{u_2,u_3\}, \{u_3,v_1\},\{v_2,b_2\} \\
\mbox{from the purple path:} & \ \{x_1',x_2'\},\{u_3,v_1'\},\{v_2',b_3\}.
\end{align*}

\tikzset{
vertex/.style={circle,draw=black,fill=white,minimum size=6mm,inner sep=1pt,font=\small},
blueedge/.style={draw=blue!70!black,line width=0.8pt},
pinkedge/.style={draw=magenta!80!black,line width=0.8pt},
greenedge/.style={draw=green!55!black,line width=0.8pt},
purpleedge/.style={draw=violet!80!black,line width=0.8pt},
treeblue/.style={draw=blue!80!black,line width=1.8pt},
treepink/.style={draw=magenta!90!black,line width=1.8pt},
treegreen/.style={draw=green!65!black,line width=1.8pt},
treepurple/.style={draw=violet!90!black,line width=1.8pt}
}

\begin{figure}[htbp]
\centering

\begin{tikzpicture}

\node[vertex] (x1) at (0,2) {$x_1$};
\node[vertex] (x2) at (1.7,2) {$x_2$};

\node[vertex] (x1p) at (0,0) {$x_1'$};
\node[vertex] (x2p) at (1.7,0) {$x_2'$};

\node[vertex] (u1) at (3.5,1) {$u_1$};
\node[vertex] (u2) at (5.2,1) {$u_2$};
\node[vertex] (u3) at (6.9,1) {$u_3$};

\node[vertex] (v1) at (8.5,2.1) {$v_1$};
\node[vertex] (v2) at (10.2,2.1) {$v_2$};

\node[vertex] (zero) at (11.8,2.7) {$0$};
\node[vertex] (b2) at (11.8,1.5) {$b_2$};

\node[vertex] (v1p) at (8.5,0) {$v_1'$};
\node[vertex] (v2p) at (10.2,0) {$v_2'$};

\node[vertex] (b1) at (11.8,0.8) {$b_1$};
\node[vertex] (b3) at (11.8,-0.8) {$b_3$};

\draw[treeblue] (x1) to[bend left=20] (x2);
\draw[pinkedge] (x1) to[bend right=20] (x2);

\draw[treeblue] (x2) to[bend left=18] (u1);
\draw[pinkedge] (x2) to[bend right=12] (u1);

\draw[greenedge] (x1p) to[bend left=20] (x2p);
\draw[treepurple] (x1p) to[bend right=20] (x2p);

\draw[treegreen] (x2p) to[bend left=14] (u1);
\draw[purpleedge] (x2p) to[bend right=18] (u1);

\draw[treeblue] (u1) to[bend left=18] (u2);
\draw[pinkedge] (u1) to[bend left=5] (u2);
\draw[greenedge] (u1) to[bend right=5] (u2);
\draw[purpleedge] (u1) to[bend right=18] (u2);

\draw[blueedge] (u2) to[bend left=18] (u3);
\draw[pinkedge] (u2) to[bend left=5] (u3);
\draw[treegreen] (u2) to[bend right=5] (u3);
\draw[purpleedge] (u2) to[bend right=18] (u3);

\draw[blueedge] (u3) to[bend left=14] (v1);
\draw[treegreen] (u3) to[bend right=12] (v1);

\draw[treeblue] (v1) to[bend left=18] (v2);
\draw[greenedge] (v1) to[bend right=18] (v2);

\draw[treeblue] (v2) to[bend left=15] (zero);

\draw[treegreen] (v2) to[bend right=12] (b2);

\draw[pinkedge] (u3) to[bend left=12] (v1p);
\draw[treepurple] (u3) to[bend right=16] (v1p);

\draw[treepink] (v1p) to[bend left=18] (v2p);
\draw[purpleedge] (v1p) to[bend right=18] (v2p);

\draw[treepink] (v2p) to[bend left=12] (b1);

\draw[treepurple] (v2p) to[bend right=16] (b3);

\end{tikzpicture}

\caption{The multigraph in Figure \ref{pic-J0}: edges of spanning tree $T$
shown by thicker lines.}
\label{pic-J0-2}

\end{figure}
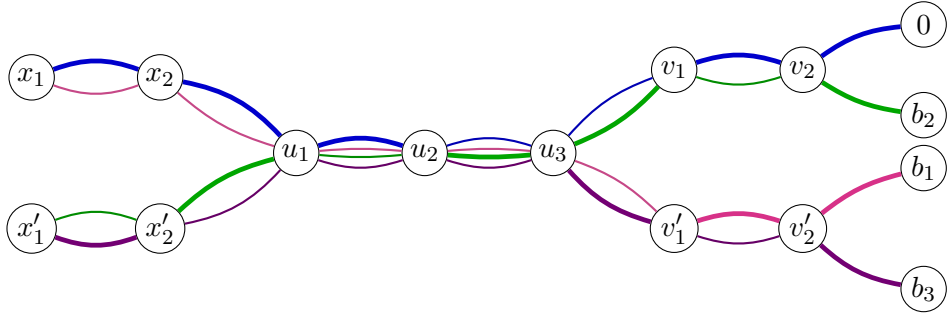

In $J_0$, we use boxes to exhibit the factors corresponding to the edges in $T$. 
\begin{align*}
&J_0(\pmb{\delta}^{(1)},\ldots,\pmb{\delta}^{(4)})
=
\int_{\bR^{14}}
d\pmb{x}_2\,d\pmb{x}_2'\,d\pmb{u}_3\,d\pmb{v}_2\,d\pmb{v}_2'\,d\pmb{b}_3\\
&\quad
\boxed{G_{\delta_7^{(1)}}(0-v_2)}
\boxed{G_{\delta_6^{(1)}}(v_2-v_1)}
G_{\delta_5^{(1)}}(v_1-u_3)
G_{\delta_4^{(1)}}(u_3-u_2)
\boxed{G_{\delta_3^{(1)}}(u_2-u_1)}
\boxed{G_{\delta_2^{(1)}}(u_1-x_2)}
\boxed{G_{\delta_1^{(1)}}(x_2-x_1)}
\\[2mm]
&\quad
\boxed{G_{\delta_7^{(2)}}(b_1-v_2')}
\boxed{G_{\delta_6^{(2)}}(v_2'-v_1')}
G_{\delta_5^{(2)}}(v_1'-u_3)
G_{\delta_4^{(2)}}(u_3-u_2)
G_{\delta_3^{(2)}}(u_2-u_1)
G_{\delta_2^{(2)}}(u_1-x_2)
G_{\delta_1^{(2)}}(x_2-x_1)
\\[2mm]
&\quad
\boxed{G_{\delta_7^{(3)}}(b_2-v_2)}
G_{\delta_6^{(3)}}(v_2-v_1)
\boxed{G_{\delta_5^{(3)}}(v_1-u_3)}
\boxed{G_{\delta_4^{(3)}}(u_3-u_2)}
G_{\delta_3^{(3)}}(u_2-u_1)
\boxed{G_{\delta_2^{(3)}}(u_1-x_2')}
G_{\delta_1^{(3)}}(x_2'-x_1')
\\[2mm]
&\quad
\boxed{G_{\delta_7^{(4)}}(b_3-v_2')}
G_{\delta_6^{(4)}}(v_2'-v_1')
\boxed{G_{\delta_5^{(4)}}(v_1'-u_3)}
G_{\delta_4^{(4)}}(u_3-u_2)
G_{\delta_3^{(4)}}(u_2-u_1)
G_{\delta_2^{(4)}}(u_1-x_2')
\boxed{G_{\delta_1^{(4)}}(x_2'-x_1')}.
\end{align*}

{We drop the factors which are not in a box, using the fact that $G_t \in L^{\infty}(\bR)$. Then we pick leaves from this tree, and we integrate with respect to them one by one, using the fact that $G_t \in L^1(\bR)$, a leaf being a variable which appears only once. We make sure that the last leaf we pick is $v_{n-r}$, which is linked to the ``root'' 0 (which is not an integration variable), and corresponds to the factor $G_{\delta_{n+1}^{(1)}}(0-v_{n-r})$. }

For this particular example and spanning tree $T$, relation \eqref{J0-bound} holds with:
\[
I_1=\{1,2,3,6,7\}, \quad I_2=\{6,7\}, \quad I_3=\{2,4,5,7\}, \quad I_4=\{1,5,7\}.
\]

We leave the example now, and return to the general argument. By \eqref{J0-bound} and \eqref{bound-G},
\[
J_0(\pmb{\delta}^{(1)},\ldots,\pmb{\delta}^{(4)}) \leq 
\left\{
\begin{array}{ll} 
\prod_{i=1}^4 \prod_{j \in J_i} (2 \pi \delta_i^{(j)})^{-1/2} & \mbox{for the heat equation} \\
\prod_{i=1}^4  (1/2)^{|J_i|} \prod_{j \in I_i} \delta_i^{(j)} & \mbox{for the wave equation}
\end{array} 
\right..
\]
Hence,
\[
J_0 \leq \prod_{i=1}^4 \int_{D_n(t)} \prod_{j=1}^{n} (\delta_j^{(i)})^{\beta_j^{(i)}} d\pmb{\delta}^{(i)}
\]
where 
\[
\beta_j^{(i)}=
\left\{
\begin{array}{ll} 
-1/2 & \mbox{if $i \in J_i$} \\
0 & \mbox{if $j \in I_i$}
\end{array} 
\right. \quad \mbox{for the heat equation},
\]
and
\[
\beta_j^{(i)}=
\left\{
\begin{array}{ll} 
0 & \mbox{if $i \in J_i$} \\
1 & \mbox{if $j \in I_i$}
\end{array} 
\right. \quad \mbox{for the wave equation}.
\]
For both equations, it follows that \fbox{$J_0<\infty$}, using the fact that for any $\beta_1,\ldots,\beta_n>-1$, 
\begin{equation}
\label{delta-int}
\int_{D_n(t)} \prod_{j=1}^n \delta_j^{\beta_j} d\delta_1 \ldots d\delta_n= \frac{\prod_{j=1}^n \Gamma(\beta_j+1)t^{|\beta|+n}}{\Gamma(|\beta|+n+1)} \quad \mbox{with} \quad |\beta|=\sum_{j=1}^n \beta_j.
\end{equation}

\underline{We now return to the case of general permutations.}
The same method that was used above for showing that $J_0<\infty$ can also be used for showing that \eqref{J-rho} holds, for any permutations $\rho_1,\ldots,\rho_4 \in S_n$.  But the multigraph becomes more complicated. Consider again the example with $n=7$, $\ell=2$ and $r=5$. Figure \ref{pic-J0-3} gives this multigraph, which correspond to permutations $\rho_1,\ldots,\rho_4 \in S_7$ leading to the following paths:
\begin{align*}
\mbox{blue path:} & \ \ \ x_1 \to v_1 \to u_3 \to u_2 \to x_2 \to v_2 \to u_1 \to 0 \\
\mbox{pink path:} & \ \ \ u_3 \to v_2' \to x_1 \to u_2 \to v_1' \to u_1 \to x_2 \to b_1\\
\mbox{green path:} & \ \ \ v_1 \to u_2 \to x_1' \to u_3 \to x_2' \to u_1 \to v_2 \to b_2\\
\mbox{purple path:} & \ \ \ u_3 \to x_1'\to u_2 \to v_2' \to u_1 \to x_2' \to v_1' \to b_3.
\end{align*}

\begin{figure}[htbp]
    \centering
    \includegraphics[width=0.65\textwidth]{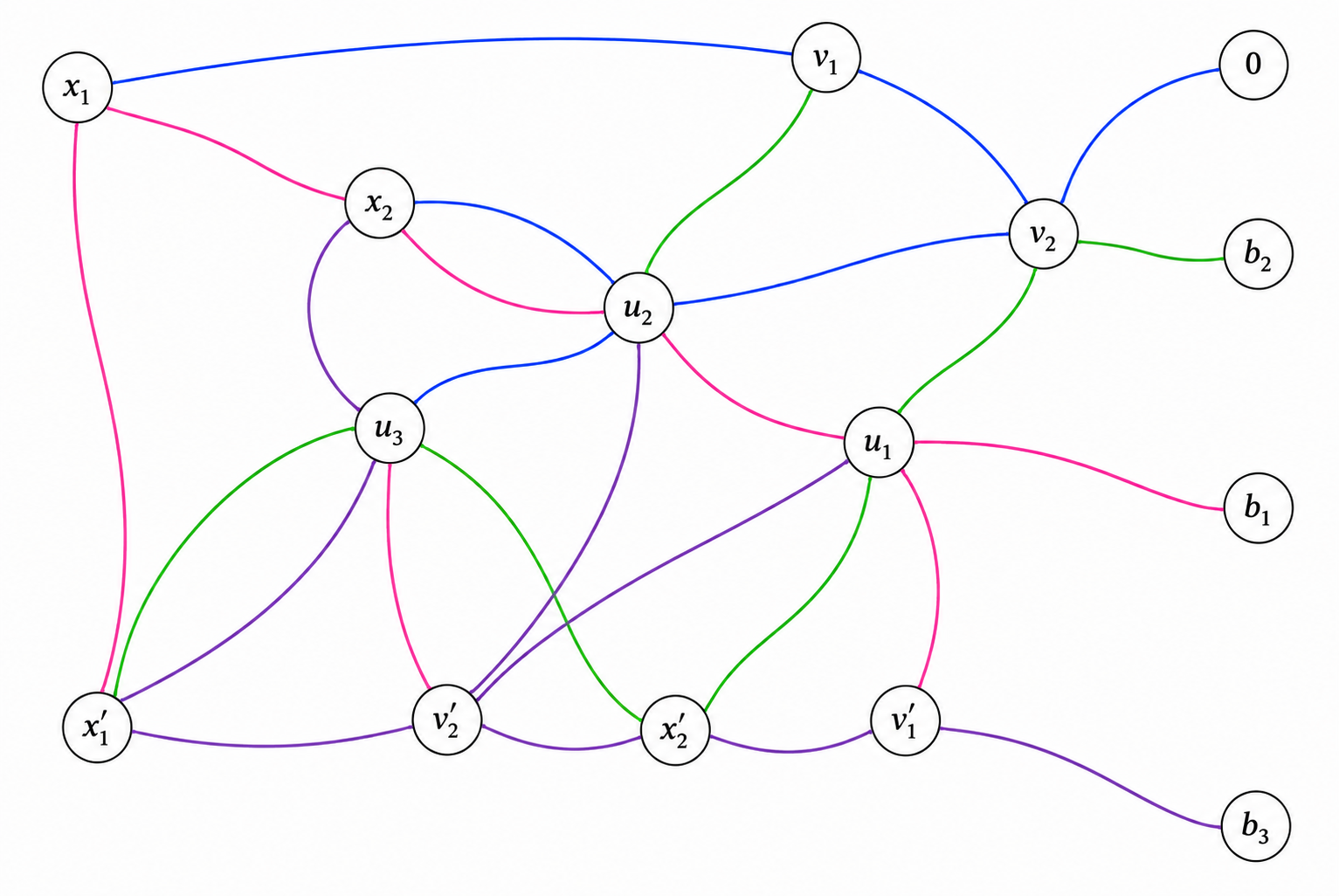}
    \caption{The multigraph for some general permutations $\rho_1,\ldots,\rho_4$.}
    \label{pic-J0-3}
\end{figure}

We omit the details.
This concludes the proof of the theorem in Case 1.

\bigskip

{\bf Case 2. (colored noise)} We fix some permutations $\rho_1,\ldots,\rho_4 \in S_n$. We come back to definition \eqref{J-rho-def} of $J_{\rho_1,\ldots,\rho_4}$ and we introduce new notation:
\begin{align*}
(\pmb{x}_{\ell},\pmb{u}_{r-\ell},\pmb{v}_{n-r})& =(w_1^{(1)},\ldots,w_n^{(1)})=:\pmb{w}^{(1)}, \quad w_{n+1}^{(1)}=0 \\
(\pmb{x}_{\ell},\pmb{u}_{r-\ell},\pmb{v}_{n-r}')& =(w_1^{(2)},\ldots,w_n^{(2)})=:\pmb{w}^{(2)}, \quad w_{n+1}^{(2)}=b_1 \\
(\pmb{x}_{\ell}',\pmb{u}_{r-\ell},\pmb{v}_{n-r})& =(w_1^{(3)},\ldots,w_n^{(3)})=:\pmb{w}^{(3)}, \quad w_{n+1}^{(3)}=b_2 \\
(\pmb{x}_{\ell}',\pmb{u}_{r-\ell},\pmb{v}_{n-r}')& =(w_1^{(4)},\ldots,w_n^{(4)})=:\pmb{w}^{(4)}, \quad w_{n+1}^{(4)}=b_3.
\end{align*}
Hence,
\begin{align*}
J_{\rho_1,\ldots,\rho_4}&= \int_{(\bR^d)^{4n+3}} \prod_{i=1}^4 \varphi_{n,\rho_i}(\pmb{w}^{(i)},w_{n+1}^{(i)};t)\, 1_{S}(\pmb{w}^{(1)},\pmb{w}^{(2)},\pmb{w}^{(3)},\pmb{w}^{(4)})
\prod_{i=1}^4 d\pmb{w}^{(i)} \prod_{i=2}^4 dw_{n+1}^{(i)},
\end{align*}
where $S$ is the set where the following property hold:
\begin{equation}
\left.\begin{array}{rcl}
& & \displaystyle  w_j^{(1)}=w_j^{(2)}=:x_j  \quad \mbox{and} \quad w_j^{(3)}=w_j^{(4)}=:x_j'  \quad \quad \quad \mbox{for all} \  j=1,\ldots,\ell, \\
& & \displaystyle  w_j^{(1)}=w_j^{(2)}=w_j^{(3)}=w_j^{(4)}=:u_{j-\ell}  \quad\quad \quad \quad \quad \quad \ \ \ \mbox{for all} \  j=\ell+1,\ldots,r, \\
& & \displaystyle  w_j^{(1)}=w_j^{(3)}=:v_{j-r}  \quad \mbox{and} \quad w_j^{(2)}=w_j^{(4)}=:v_{j-r}'  \quad \mbox{for all} \  j=r+1,\ldots,n.
\end{array} \right\} \label{propertyP} 
\end{equation}

By definition \eqref{def-phi-n} and the change of variable $\delta_j=t_{j+1}-t_j$ for $j\leq n-1$, $\delta_n=t-t_{n}$,
\begin{align*}
\varphi_n(\pmb{x},x;t)&=\int_{D_n(t)} \int_{(\bR^d)^n}\prod_{j=1}^{n-1} G_{\delta_j}(\widetilde{x}_{j+1}-\widetilde{x}_j)  G_{\delta_n}(x-\widetilde{x}_n)  \prod_{j=1}^n \k(x_j-\widetilde{x}_j) d\widetilde{\pmb{x}} d\pmb{\delta},
\end{align*}
where $\pmb{x}=(x_1,\ldots,x_n)$, $\widetilde{\pmb{x}}=(\widetilde{x}_1,\ldots,\widetilde{x}_n)$ and $\pmb{\delta}=(\delta_1,\ldots,\delta_n)$. Therefore,
\begin{align*}
\varphi_{n,\rho}(\pmb{w},w_{n+1};t)&=\varphi_{n}(w_{\rho(1)}, \ldots,w_{\rho(n+1)},w_{n+1};t)\\
&=\int_{D_n(t)} \int_{(\bR^d)^n}  \prod_{j=1}^{n-1} G_{\delta_j}(\widetilde{w}_{j+1}-\widetilde{w}_j) G_{\delta_n}(w_{n+1}-\widetilde{w}_n) \prod_{j=1}^n \k(w_{\rho(j)}-\widetilde{w}_j) d \pmb{\widetilde{w}} d\pmb{\delta},
\end{align*}
where  $\pmb{w}=(w_1,\ldots,w_n)$ and $\widetilde{\pmb{w}}=(\widetilde{w}_1,\ldots,\widetilde{w}_n)$. It follows that \eqref{J-rho14} holds with
\begin{align}
\nonumber
J_{\rho_1,\ldots,\rho_4} (\pmb{\delta}^{(1)}, \ldots, \pmb{\delta}^{(4)})&=\int_{(\bR^d)^{8n+3}}  \prod_{i=1}^4 d\pmb{w}^{(i)} \prod_{i=1}^4d\widetilde{\pmb{w}}^{(i)} \prod_{i=2}^4 dw_{n+1}^{(i)} 1_{S}(\pmb{w}^{(1)},\pmb{w}^{(2)},\pmb{w}^{(3)},\pmb{w}^{(4)}) \\
\label{def-J-delta}
& \prod_{i=1}^4 \left( \prod_{j=1}^{n-1} G_{\delta_j^{(i)}}(\wt_{j+1}^{(i)}-\wt_j^{(i)}) G_{\delta_n^{(i)}}(w_{n+1}^{(i)}-\wt_n^{(i)}) \prod_{j=1}^n \k(w_{\rho_i(j)}^{(i)}-\wt_j^{(i)}) \right),
\end{align}
where $\widetilde{\pmb{w}}^{(i)}=(\wt_1^{(i)},\ldots,\wt_n^{(i)})$ for $i=1,\ldots,4$.

To calculate $J_{\rho_1,\ldots,\rho_4}(\pmb{\delta}^{(1)}, \ldots, \pmb{\delta}^{(4)})$, we introduce the graph 
\begin{equation}
\label{def-Gcol}
G_{\rm colored}=(V_{\rm colored},E_{\rm colored}) \quad \mbox{with} \quad V_{\rm colored}=\bigcup_{i=1}^4 V_i, \quad E_{\rm colored}=\bigcup_{i=1}^4 E_i,
\end{equation}
where
\begin{align*}
V_i & = \{w_1^{(i)},\ldots,w_n^{(i)},w_{n+1}^{(i)},\wt_{1}^{(i)},\ldots,\wt_n^{(i)}\}  \\
E_i & = \bigcup_{j=1}^{n-1} \{\wt_j^{(i)},\wt_{j+1}^{(i)}\} \bigcup \{\wt_n^{(i)},w_{n+1}^{(i)}\} \bigcup \bigcup_{j=1}^{n} \{w_{\rho_i(j)}^{(i)}, \wt_j^{(i)} \}.
\end{align*}
The vertices of $G_{\rm colored}$ are the $8n+3$ integration variables appearing in $J_{\rho_1,\ldots,\rho_4} (\pmb{\delta}^{(1)}, \ldots, \pmb{\delta}^{(4)})$, plus the vertex $w_{n+1}^{(1)}=0$. 
We connect two vertices $u$ and $v$ with an edge, if the factor $G_{\delta_j^{(i)}}(u-v)$ or the factor $\k(u-v)$ appears in $J_{\rho_1,\ldots,\rho_4} (\pmb{\delta}^{(1)}, \ldots, \pmb{\delta}^{(4)})$. 

Note that the
vertices $\wt_j^{(i)}$ are all distinct, but the $w_j^{(i)}$'s are not all distinct, due to property \eqref{propertyP}. More precisely,
${\rm card}(V)=4(2n+1)-2\ell-3(r-\ell)-2(n-r)$.

\medskip

We consider separately Cases (ii) and (iii). We will select a spanning tree of $G_{\rm colored}$ using a different method for each case.  As in Case 1, we use boxes to exhibit the factors in \eqref{def-J-delta} which correspond to an edge in the tree, and we drop the factors which are not in a box. Then we pick leaves from this tree, and we integrate with respect to them one by one. (A leaf is a variable which appears only once.)

\begin{itemize}
\item In \underline{Case (ii)}, we use Lemma \ref{tree-d-equal1} to select a spanning tree  which contains all edges of the form $\{w_{\rho_i(j)}^{(i)},\wt_j^{(i)}\}$, as well as the edge $\{\wt_n^{(1)},w_{n+1}^{(1)}\}$. 
All the $k$-factors 
and $G_{\delta_n^{(1)}}(\wt_n^{(1)},w_{n+1}^{(1)})$ will be in a box, and will be kept. 

{\em Step 1.} The factors which are not in a box are all $G$'s. These are dropped, using the fact that $G_t \in L^{\infty}(\bR)$ (for both heat and wave equation in dimension $d=1$). This produces some constants of the form $\big(2\pi \delta_j^{(i)}\big)^{-1/2}$ for the heat equation, or $1/2$ for the wave equation. The remaining factors are all the $\k$'s, and some $G$'s. 

{\em Step 2.} We pick leaves from the tree, and we integrate with respect to them one by one, using the fact that $\k \in L^1(\bR)$ and $G_t \in L^1(\bR)$. We make sure that the last leaf we pick is $\wt_{n+1}^{(1)}$, which is linked to the $w_{n+1}^{(1)}=0$, and corresponds to the factor $G_{\delta_{n+1}^{(1)}}(0-w_{n+1}^{(1)})$, which is integrated last (because $w_{n+1}^{(1)}=0$  is not an integration variable). This produces some constants of the form $ \delta_j^{(i)}$ for the wave equation. 

{\em Step 3.} The integral with respect to $\delta_j^{(i)}$'s of the constants produced in Step 1 and Step 2 is finite, due to \eqref{delta-int}.  Recalling \eqref{J-rho14}, we conclude that $J_{\rho_1,\ldots,\rho_4}<\infty$.

\item In \underline{Case (iii)}, we use Lemma \ref{tree-d-equal2} to select a spanning tree  which contains  all edges of the form $\{\wt_{j}^{(i)},\wt_{j+1}^{(i)}\}$, as well as $\{\wt_n^{(i)},w_{n+1}^{(i)}\}$. All the $G$-factors will be in a box, and will be kept. 
    
 {\em Step 1.}  The factors which are not in a box are all $k$'s. These are dropped, using the fact that $\k \in L^{\infty}(\bR^2)$. The remaining factors are all the $G$'s, and some $\k$'s.
 
{\em Steps 2-3.} These steps are exactly as in Case (ii).

\end{itemize}

\end{proof}

\begin{lemma}
\label{connected}
The graph $G_{\rm colored}$ given by \eqref{def-Gcol} is connected.
\end{lemma}

\begin{proof}
For $i$ fixed, we define the subgraph $G_i=(V_i,E_i)$ as the restriction of $G$ with vertices of the form $\wt_{j}^{(i)}$ or $w_{\rho_i(j)}^{(i)}$ and the edges connecting them (see Figure ~\ref{Key} below). For each $i=1,2,3,4$, these subgraphs are connected.

	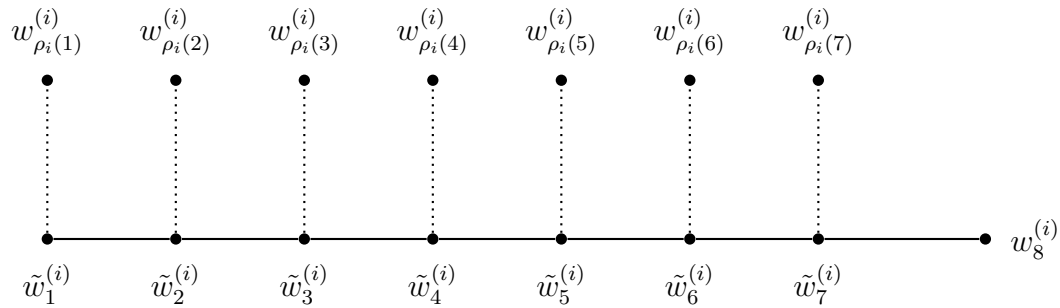
\begin{figure}[ht]
		\centering
		\begin{tikzpicture}[
			x=1.7cm,
			y=1.5cm,
			point/.style={circle, fill=black, inner sep=1.5pt},
			rail/.style={thick},
			spoke/.style={thick, dotted}
			]
			
			\def\perm{{3,1,6,2,7,4,5}}
			
			\def\n{7}
			
			\foreach \j in {1,...,\n} {
				\node[point] (X\j) at (\j,0) {};
				\node[below=3pt of X\j] {$\tilde{w}_{\j}^{(i)}$};
			}
			
			\foreach \j [evaluate=\j as \next using int(\j+1)] in {1,...,6} {
				\draw[rail] (X\j) -- (X\next);
			}
			
			\foreach \j in {1,...,\n} {
				\pgfmathtruncatemacro{\rhovalue}{\perm[\j-1]}
				
				\node[point] (W\j) at (\j,1.4) {};
				\node[above=3pt of W\j]
				{$w_{\rho_i(\j)}^{(i)}$};
				
				\draw[spoke] (X\j) -- (W\j);
			}
			
			\node[point] (Wend) at (8.3,0) {};
			\node[right=3pt of Wend] {$w_{8}^{(i)}$};
			
			\draw[rail] (X7) -- (Wend);
			
		\end{tikzpicture}
		\caption{The subgraph $G_i$ for $n=7$. The solid and dotted lines correspond to the functions $G$ and $\kappa$ in the integral, respectively.}
		\label{Key}
	\end{figure}

 We consider different cases to show that $G$ is connected. 
	
	If $\ell=0$, then there are $a,b,c,d\in \{1,2,...,n\}$ such that 
	\begin{equation*}
		\rho_1(a)=\rho_2(b)=\rho_3(c)=\rho_4(d)=1,
	\end{equation*}
	and 
	\begin{equation*}
		w_{\rho_1(a)}^{(1)}=w_{\rho_2(b)}^{(2)}=w_{\rho_3(c)}^{(3)}=w_{\rho_4(d)}^{(4)}{=u_1}.
	\end{equation*}
	Since the subgraphs $G_i$ have a common vertex and are themselves connected we conclude that $G$ is connected as well. 
	
	If $0<\ell<r$, we apply the same argument as in the $\ell=0$ case (instead of using the number 1 use $\ell+1$).
	
	If $0<\ell=r<n$, we start by noticing that the subgraph $G_1\cup G_2$ is connected since there are $a,b\in \{1,2,...,n\}$
	\begin{equation*}
		1=\rho_1(a)=\rho_2(b) \text{ and therefore } w_{\rho_1(a)}^{(1)}=w_{\rho_2(b)}^{(2)}{=x_1}.
	\end{equation*} 
	In the exact same way we can also show that $G_3\cup G_4$ is connected. Finally, $G_1\cup G_3$ is connected since there are $c,d\in\{1,2,...,n\}$ such that 
	\begin{equation*}
		\ell+1=\rho_1(c)=\rho_3(d) \text{ and therefore } x_{\rho_1(c)}^{(1)}=x_{\rho_3(d)}^{(3)}{=x_1'}.
	\end{equation*}
	Therefore, we can conclude that $G=G_1\cup G_2\cup G_3 \cup G_4$ is connected.

\end{proof}

\begin{lemma}
\label{tree-d-equal1}
The graph $G_{\rm colored}$ has a spanning tree which contains all edges of the form $\{w_{\rho_i(j)}^{(i)},\wt_j^{(i)}\}$ for $i=1,\ldots,4$ and $j=1,\ldots,n$, as well as 
$\{\wt_n^{(1)},w_{n+1}^{(1)}\}$.
\end{lemma}

\begin{proof}
First, we explain the representation of the graph $G_{\rm colored}$, which will be used for constructing the tree. The graph has edges of the form $\{\wt_{j}^{(i)},w_{\rho_{i}(j)}^{(i)}\}$ for $j=1,\ldots,n$, which can be listed also as $\{w_{j}^{(i)},\wt_{\rho_i^{-1}(j)}^{(i)}\}$ for $j=1,\ldots,n$. Because of property \eqref{propertyP} of the nodes, we identify the following disjoint connected components of the graph:\\
-subgraph $T_j$ with center node $w_j^{(1)}=w_j^{(2)}=x_j$, for $j=1,\ldots,\ell$;\\
-subgraph $T_j'$ with center node $w_j^{(3)}=w_j^{(4)}=x_j'$, for $j=1,\ldots,\ell$;\\
-subgraph $S_j$ with center node $w_j^{(1)}=w_j^{(2)}=w_j^{(3)}=w_j^{(4)}=u_{j-\ell}$, for $j=\ell+1,\ldots,r$;\\
-subgraph $Q_j$ with center node $w_j^{(1)}=w_j^{(3)}=v_{j-r}$, for $j=r+1,\ldots,n$;\\
-subgraph $Q_j'$ with center node $w_j^{(2)}=w_j^{(4)}=v_{j-r}'$, for $j=r+1,\ldots,n$.

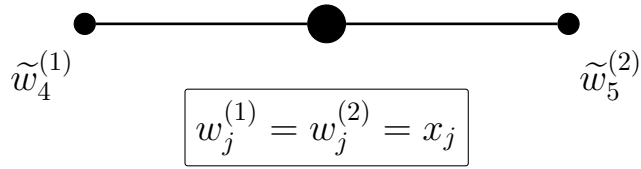
\begin{figure}[htbp]
\centering



\tikzset{
  smalldot/.style={
    circle,
    draw=black,
    fill=black,
    minimum size=2.8mm,
    inner sep=0pt
  }
}

\tikzset{
  centerdot/.style={
    circle,
    draw=black,
    fill=black,
    minimum size=5mm,
    inner sep=0pt
  }
}

\tikzset{
  lab/.style={
    font=\large,
    inner sep=1pt
  }
}

\tikzset{
  eqlab/.style={
    font=\large,
    draw=black,
    rounded corners=1pt,
    fill=white,
    inner sep=4pt
  }
}

\begin{tikzpicture}[x=1cm,y=1cm]

\node[smalldot]  (left)   at (0,0) {};
\node[centerdot] (center) at (3.2,0) {};
\node[smalldot]  (right)  at (6.4,0) {};

\draw[black,line width=1pt]
  (left)--(center)--(right);

\node[lab,below left=2mm and 0mm of left]
  {$\widetilde w_{4}^{(1)}$};

\node[eqlab,below=6mm of center]
  {$w_j^{(1)}=w_j^{(2)}=x_j$};

\node[lab,below right=2mm and 0mm of right]
  {$\widetilde w_{5}^{(2)}$};

\end{tikzpicture}

\caption{Subgraph $T_j$ for $j=1,\ldots,\ell$.}
\label{fig-Tj}
\end{figure}

\begin{figure}[htbp]
\centering



\tikzset{smalldot/.style={
  circle,
  draw=black,
  fill=black,
  minimum size=2.8mm,
  inner sep=0pt
}}

\tikzset{centerdot/.style={
  circle,
  draw=black,
  fill=black,
  minimum size=5mm,
  inner sep=0pt
}}

\tikzset{lab/.style={
  font=\large,
  inner sep=1pt
}}

\tikzset{eqlab/.style={
  font=\large,
  draw=black,
  rounded corners=1pt,
  fill=white,
  inner sep=4pt
}}

\begin{tikzpicture}[x=1cm,y=1cm]

\node[centerdot] (c) at (0,0) {};
\node[smalldot]  (t) at (0,1.6) {};
\node[smalldot]  (b) at (0,-1.6) {};
\node[smalldot]  (l) at (-2.2,0) {};
\node[smalldot]  (r) at (2.2,0) {};

\draw[black,line width=1pt] (c)--(t);
\draw[black,line width=1pt] (c)--(b);
\draw[black,line width=1pt] (c)--(l);
\draw[black,line width=1pt] (c)--(r);

\node[lab,above=2mm of t] {$\widetilde w_{\rho_1^{-1}(j)}^{(1)}$};
\node[lab,below=2mm of b] {$\widetilde w_{\rho_3^{-1}(j)}^{(3)}$};
\node[lab,left=2mm of l] {$\widetilde w_{\rho_4^{-1}(j)}^{(4)}$};
\node[lab,right=2mm of r] {$\widetilde w_{\rho_2^{-1}(2)}^{(2)}$};

\node[eqlab,anchor=north west] at (0.55,-0.45)
  {$w_j^{(1)}=w_j^{(2)}=w_j^{(3)}=w_j^{(4)}=u_{j-\ell}$};

\end{tikzpicture}

\caption{Subgraph $S_j$ for $j=\ell+1,\ldots,r$.}
\label{fig-Sj}
\end{figure}
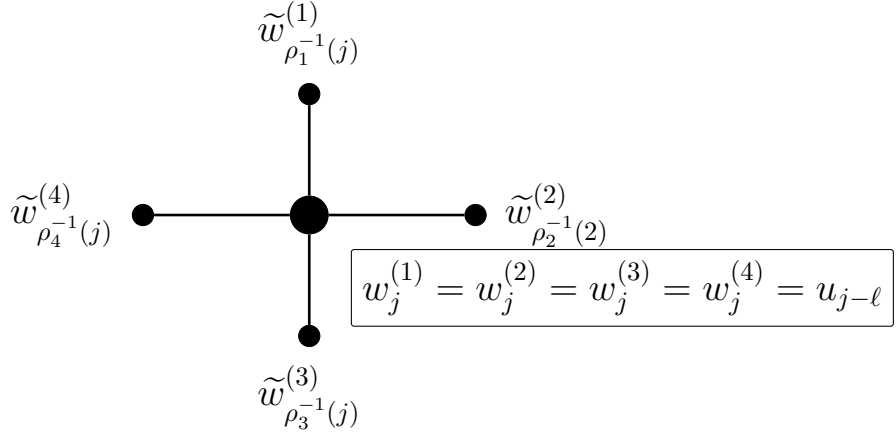

Figures \ref{fig-Tj} and \ref{fig-Sj} show the subgraphs $T_j$, respectively $S_j$. The subgraphs $T_j'$, $Q_j$ and $Q_j'$ have the same shape as $T_j$.

To represent $G_{\rm colored}$, we draw these $2\ell+(r-\ell)+2(n-r)$ subgraphs, which correspond to the $\k$-factors in the product \eqref{def-J-delta}, then we draw the edges connecting $\wt_1^{(i)},\ldots,\wt_{n}^{(i)},w_{n+1}^{(i)}$ for $i=1,\ldots,4$, which correspond to the $G$-factors in \eqref{def-J-delta}.

\medskip

We now explain the construction of the tree.

{\em Step 0.} The tree includes all the subgraphs $T_j,T_j',S_j,Q_j,Q_j'$ mentioned above.

{\em Step 1.} We connect the vertices $\wt_1^{(1)},\ldots, \wt_n^{(1)},w_{n+1}^{(1)}$. This produces connected component $C_1$ which includes:\\
- the subgraphs $T_1,\ldots,T_{\ell}$, with respective centers at $\pmb{x}_{\ell}=(x_1,\ldots,x_{\ell})$;
\\
- the subgraphs $S_{\ell+1},\ldots,S_{r}$, with respective centers at $\pmb{u}_{r-\ell}=(u_1,\ldots,u_{r-\ell})$; \\
- the subgraphs $Q_{r+1},\ldots,Q_{\ell}$, with respective centers at $\pmb{v}_{n-r}=(v_1,\ldots,v_{n-r})$.\\
It remains to connect $T_1',\ldots,T_{\ell}',Q_{r+1}',\ldots,Q_n'$ and $w_{n+1}^{(2)},w_{n+1}^{(3)},w_{n+1}^{(4)}$ to $C_1$.

{\em Step 2.} Pick one vertex $\wt_j^{(i)}$ from $T_1'$ and connect it to $C_1$. Such an edge exists, since the graph $G_{\rm colored}$ is connected, by Lemma \ref{connected}. This does not create a cycle. Repeat this procedure for $T_{2}',\ldots, T_{\ell}'$. This produces connected component $C_2$, which includes in addition the subgraphs $T_1',\ldots,T_{\ell}'$ with respective centers at $\pmb{x}_{\ell}'=(x_1',\ldots,x_{\ell}')$. It remains to connect $Q_{r+1}',\ldots,Q_n'$ and $w_{n+1}^{(2)},w_{n+1}^{(3)},w_{n+1}^{(4)}$ to $C_2$.

{\em Step 3.}  Pick one vertex $\wt_j^{(i)}$ from $Q_{r+1}'$ and connect it to $C_2$. Repeat this procedure for $Q_{r+2}',\ldots, Q_{n}'$. This produces connected component $C_3$ which includes in addition the subgraphs $Q_{r+1}',\ldots,Q_{\ell}'$, with respective centers at $\pmb{v}_{n-r}'=(v_1',\ldots,v_{n-r}')$. It remains to connect $w_{n+1}^{(2)},w_{n+1}^{(3)},w_{n+1}^{(4)}$ to $C_3$.

{\em Step 4.} Connect $w_{n+1}^{(2)},w_{n+1}^{(3)},w_{n+1}^{(4)}$ to $C_3$.

\end{proof}

\begin{example}
\label{ex-col1}
{\rm
Let $n=7$, $\ell=2$ and $r=5$. Figure \ref{fig-G-colored1} gives a representation of graph $G_{\rm colored}$ corresponding to the following permutations:
\begin{equation}
\label{def-rho12}
\rho_1=
\begin{pmatrix}
1 & 2 & 3 & 4 & 5 & 6 & 7\\
3 & 7 & 2 & 1 & 4 & 5 & 6
\end{pmatrix}
\qquad
\rho_2=
\begin{pmatrix}
1 & 2 & 3 & 4 & 5 & 6 & 7\\
2 & 3 & 5 & 7 & 1 & 4 & 6
\end{pmatrix}
\end{equation}
\begin{equation}
\label{def-rho34}
\rho_3=
\begin{pmatrix}
1 & 2 & 3 & 4 & 5 & 6 & 7\\
7 & 4 & 5 & 2 & 6 & 1 & 3
\end{pmatrix},
\qquad
\rho_4=
\begin{pmatrix}
1 & 2 & 3 & 4 & 5 & 6 & 7\\
5 & 6 & 1 & 3 & 7 & 4 & 2
\end{pmatrix}
\end{equation}
with inverses:
\[
\rho_1^{-1}=
\begin{pmatrix}
1 & 2 & 3 & 4 & 5 & 6 & 7\\
4 & 3 & 1 & 5 & 6 & 7 & 2
\end{pmatrix},
\qquad
\rho_2^{-1}=
\begin{pmatrix}
1 & 2 & 3 & 4 & 5 & 6 & 7\\
5 & 1 & 2 & 6 & 3 & 7 & 4
\end{pmatrix},
\]
\[
\rho_3^{-1}=
\begin{pmatrix}
1 & 2 & 3 & 4 & 5 & 6 & 7\\
6 & 4 & 7 & 2 & 3 & 5 & 1
\end{pmatrix},
\qquad
\rho_4^{-1}=
\begin{pmatrix}
1 & 2 & 3 & 4 & 5 & 6 & 7\\
3 & 7 & 4 & 6 & 1 & 2 & 5
\end{pmatrix}.
\]

\begin{figure}[htbp]
    \centering
    \input{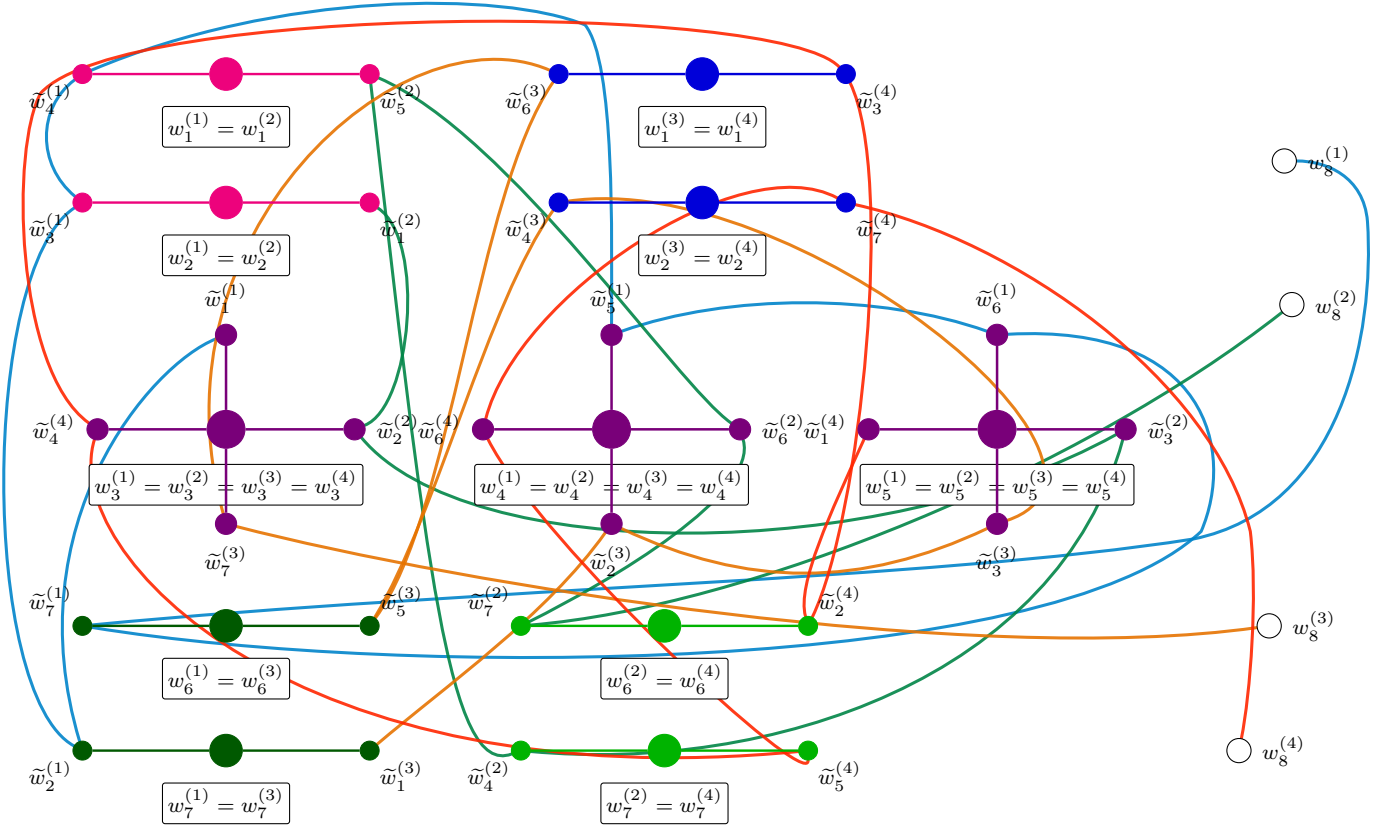}
    \caption{The graph $G_{\rm colored}$ for $n=7,\ell=2,r=5$ and $\rho_i$ given by \eqref{def-rho12} and \eqref{def-rho34}.}
    \label{fig-G-colored1}
\end{figure}

The graph in Figure \ref{fig-G-colored1} is constructed as follows. We first draw the subgraphs $T_1$, $T_2$, $T_1'$, $T_2'$ (lines at the top), $S_1$, $S_2$, $S_3$ (crosses in the middle), and $Q_1$, $Q_2$, $Q_1'$, $Q_2'$ (lines at the bottom), using the inverse permutations. For instance, for $T_1$ (top left, in magenta), the node $w_1^{(1)}=w_1^{(2)}$ is connected to $\wt_4^{(1)}$ on the left because $\rho_1^{-1}(1)=4$, and to  $\wt_5^{(2)}$ on the right because $\rho_2^{-1}(1)=5$. This introduces the edges $\{w_{j}^{(i)},\wt_{\rho_i^{-1}(j)}^{(i)}\}$ for $j=1,\ldots,7$ which can be listed also as $\{\wt_{j}^{(i)},w_{\rho_{i}(j)}^{(i)}\}$ for $j=1,\ldots,7$. These edges represent the product of the $k$-factors in \eqref{def-J-delta}, and will be kept for the construction of the tree. Then, we draw the four curved lines connecting $\wt_1^{(i)}, \ldots, \wt_7^{(i)}$ and $w_8^{(i)}$ using different colors: blue for $i=1$, green for $i=2$, orange for $i=3$, and red for $i=4$. These edges represent the product of the $G$-factors in \eqref{def-J-delta}.

Figure \ref{fig-G-colored1-tree} gives the tree from Lemma \ref{tree-d-equal1} with the edges constructed in each step.
\begin{figure}[htbp]
   \centering
   \input{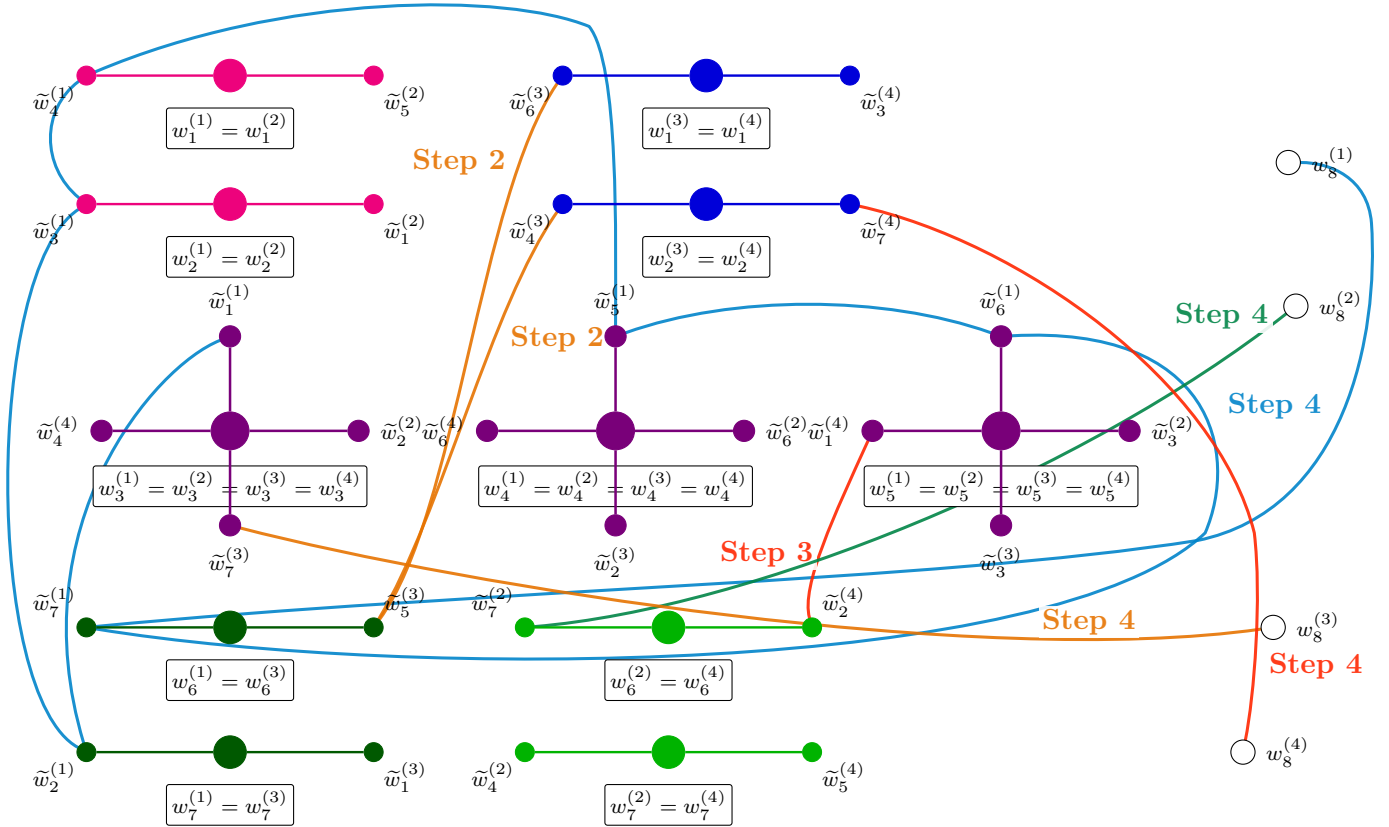}
    \caption{The spanning tree from Lemma \ref{tree-d-equal1}.}
    \label{fig-G-colored1-tree}
\end{figure}

}
\end{example}

\begin{lemma}
\label{tree-d-equal2}
The graph $G_{\rm colored}$ has a spanning tree which contains all edges of the form $\{\wt_j^{(i)},\wt_{j+1}^{(i)}\}$ for  $i=1,\ldots,4$ and $j=1,\ldots,n-1$, as well as $\{\wt_n^{(i)},w_{n+1}^{(i)}\}$ for $i=1,\ldots,4$.
\end{lemma}

\begin{proof}
We draw the graph $G_{\rm colored}$ differently than in Lemma \ref{tree-d-equal1}. First, we draw four parallel lines connecting $\wt_1^{(i)},\ldots,\wt_n^{(i)},w_{n+1}^{(i)}$, for each $i=1,\ldots,4$. These are the edges corresponding to the $G$-factors in \eqref{def-J-delta}. Recalling property \eqref{propertyP}, we put the nodes $\pmb{x}_{\ell}$, $\pmb{x}_{\ell}'$, $\pmb{u}_{r-\ell}$, $\pmb{v}_{n-r}$, $\pmb{v}'_{n-r}$ on the right side of these lines. On each line $i$, we connect $\wt_j^{(i)}$  with the node $w_{\rho_i(j)}^{(i)}$ on the right, for $j=1,\ldots,n$. These edges correspond to the $\k$-factors in \eqref{def-J-delta}.

\medskip

We now explain the construction of the tree.

{\em Step 0.} The tree includes the four parallel lines mentioned above.

{\em Step 1.} We connect each vertex $\wt_j^{(1)}$ on the first line with the node $w_{\rho_1(j)}^{(1)}$ at the right, for $j=1,\ldots,n$. This produces a connected component which contains the first line and the nodes $\pmb{x}_{\ell}$, $\pmb{u}_{r-\ell}$, $\pmb{v}_{n-r}$. 

{\em Step 2.} Pick one node from $\pmb{x}_{\ell}$ and connect it to the corresponding vertex $\wt_j^{(2)}$. This connects the first two lines. Then, connect each node from $\pmb{v}'_{n-r}$ with the corresponding vertex $\wt_j^{(2)}$ on the second line. This produces a connected component $C_1$ which contains the first two lines, and $\pmb{x}_{\ell}$, $\pmb{u}_{r-\ell}$, $\pmb{v}_{n-r}$, $\pmb{v}_{n-r}'$.

{\em Step 3.}  Connect each node from $\pmb{x}'_{\ell}$ with the corresponding vertex $\wt_j^{(3)}$ on the third line.  This produces a connected component $C_2$.

{\em Step 4.} Connect $C_1$ and $C_2$: pick a node from $\pmb{v}_{n-r}'$ or $\pmb{u}_{r-\ell}$) and connect it to the corresponding vertex $\wt_j^{(3)}$ on the third line. This produces a connected component $C_3$ which contains the first three lines, and $\pmb{x}_{\ell}$, $\pmb{x}_{\ell}'$, $\pmb{u}_{r-\ell}$, $\pmb{v}_{n-r}$, $\pmb{v}_{n-r}'$.

{\em Step 5.} Connect $C_4$ with the fourth line:  pick a node from $\pmb{x}_{\ell}'$, $\pmb{u}_{r-\ell}$ or $\pmb{v}_{n-r}'$ and connect it to the corresponding vertex $\wt_j^{(4)}$ on the fourth line. 

\end{proof}

\begin{example}
\label{ex-col2}
{\rm Let $n=7$, $\ell=2$ and $r=5$. We consider the permutations given by \eqref{def-rho12} and \eqref{def-rho34}. Figure \ref{fig-G-colored2} gives the graph $G_{\rm colored}$, using the same colors as in Example \ref{ex-col1}. 
Figure \ref{fig-G-colored2-tree} gives the spanning tree from Lemma \ref{tree-d-equal2}.

\begin{figure}[htbp]
    \centering
    \input{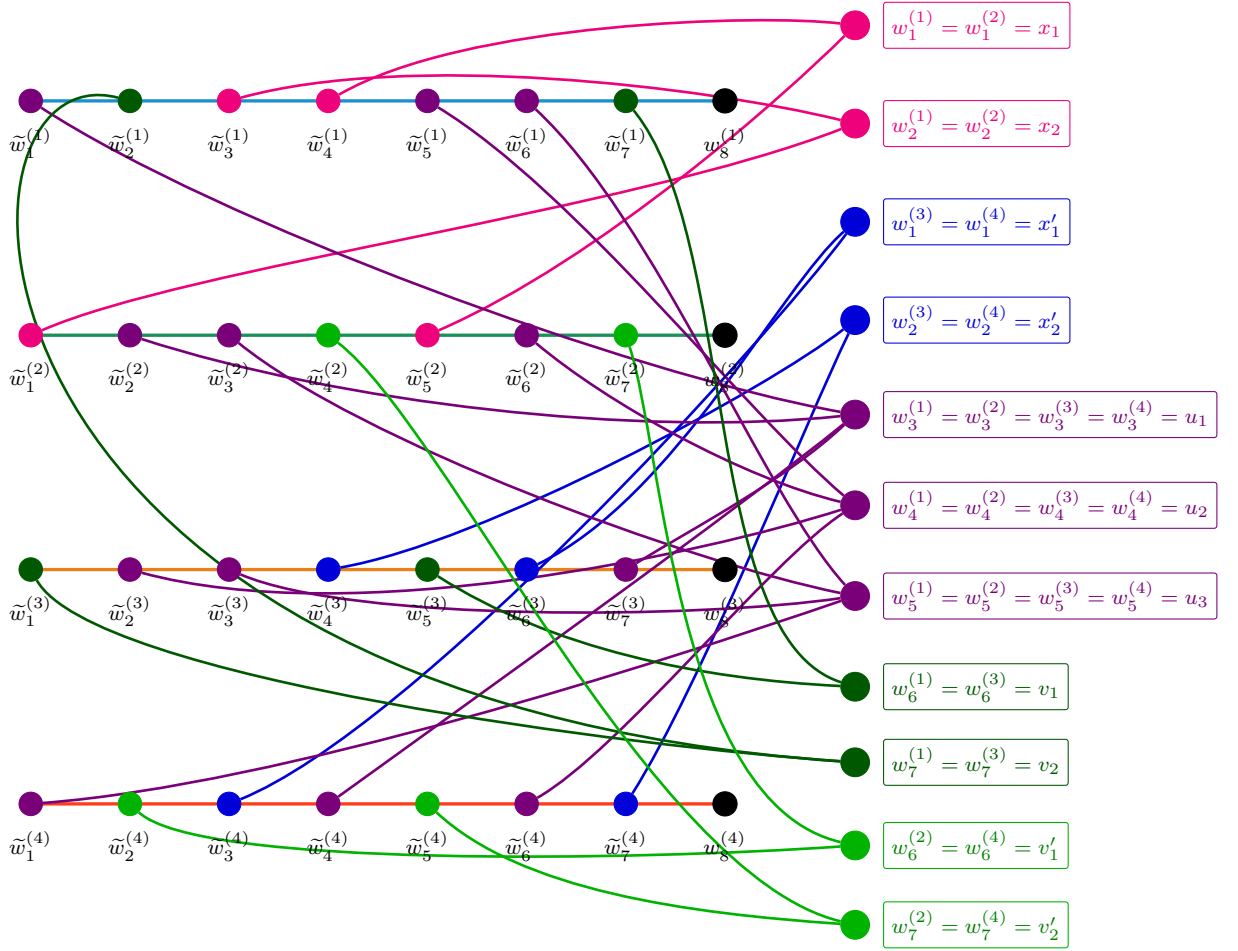}
    \caption{The graph $G_{\rm colored}$ for $n=7,\ell=2,r=5$ and $\rho_i$ given by \eqref{def-rho12} and \eqref{def-rho34}.}
    \label{fig-G-colored2}
\end{figure}

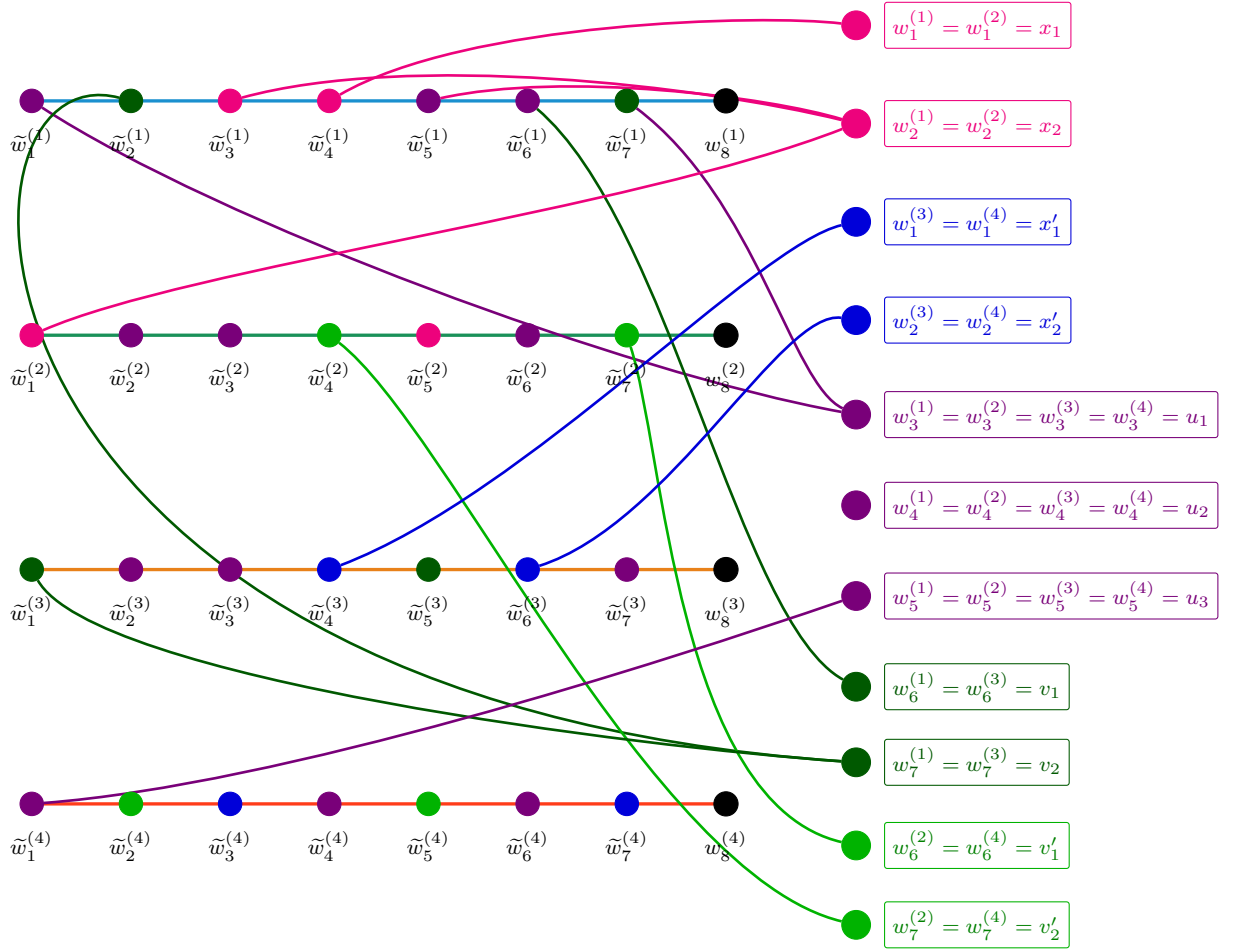
\begin{figure}[htbp]
    \centering


\tikzset{rowdot/.style={circle,minimum size=3.1mm,inner sep=0pt}}
\tikzset{common/.style={circle,minimum size=3.8mm,inner sep=0pt}}
\tikzset{terminal/.style={circle,draw=black,fill=black,minimum size=3.2mm,inner sep=0pt}}
\tikzset{lab/.style={font=\scriptsize,inner sep=1pt}}

\tikzset{magentaDot/.style={rowdot,draw=magenta!85!red,fill=magenta!85!red}}
\tikzset{blueDot/.style={rowdot,draw=blue!85!black,fill=blue!85!black}}
\tikzset{purpleDot/.style={rowdot,draw=violet!95!black,fill=violet!95!black}}
\tikzset{darkgreenDot/.style={rowdot,draw=green!35!black,fill=green!35!black}}
\tikzset{lightgreenDot/.style={rowdot,draw=green!70!black,fill=green!70!black}}

\tikzset{magentaCommon/.style={common,draw=magenta!85!red,fill=magenta!85!red}}
\tikzset{blueCommon/.style={common,draw=blue!85!black,fill=blue!85!black}}
\tikzset{purpleCommon/.style={common,draw=violet!95!black,fill=violet!95!black}}
\tikzset{darkgreenCommon/.style={common,draw=green!35!black,fill=green!35!black}}
\tikzset{lightgreenCommon/.style={common,draw=green!70!black,fill=green!70!black}}

\tikzset{magentaEdge/.style={draw=magenta!85!red,line width=1.0pt}}
\tikzset{blueEdge/.style={draw=blue!85!black,line width=1.0pt}}
\tikzset{purpleEdge/.style={draw=violet!95!black,line width=1.0pt}}
\tikzset{darkgreenEdge/.style={draw=green!35!black,line width=1.0pt}}
\tikzset{lightgreenEdge/.style={draw=green!70!black,line width=1.0pt}}

\tikzset{lightbluepath/.style={draw=cyan!65!blue,line width=1.25pt,opacity=.90,line cap=round}}
\tikzset{greencurlypath/.style={draw=teal!70!green!80!black,line width=1.25pt,opacity=.90,line cap=round}}
\tikzset{orangecurlypath/.style={draw=orange!90!black,line width=1.25pt,opacity=.90,line cap=round}}
\tikzset{redcurlypath/.style={draw=red!70!orange,line width=1.25pt,opacity=.90,line cap=round}}

\tikzset{magentaBox/.style={draw=magenta!85!red,text=magenta!85!red,rounded corners=1pt,fill=white,inner sep=3pt,font=\scriptsize}}
\tikzset{blueBox/.style={draw=blue!85!black,text=blue!85!black,rounded corners=1pt,fill=white,inner sep=3pt,font=\scriptsize}}
\tikzset{purpleBox/.style={draw=violet!95!black,text=violet!95!black,rounded corners=1pt,fill=white,inner sep=3pt,font=\scriptsize}}
\tikzset{darkgreenBox/.style={draw=green!35!black,text=green!35!black,rounded corners=1pt,fill=white,inner sep=3pt,font=\scriptsize}}
\tikzset{lightgreenBox/.style={draw=green!70!black,text=green!50!black,rounded corners=1pt,fill=white,inner sep=3pt,font=\scriptsize}}

\begin{tikzpicture}[x=0.82cm,y=1cm]


\node[purpleDot]    (r1w3) at (0.0,10.0) {};
\node[darkgreenDot] (r1w7) at (1.6,10.0) {};
\node[magentaDot]   (r1w2) at (3.2,10.0) {};
\node[magentaDot]   (r1w1) at (4.8,10.0) {};
\node[purpleDot]    (r1w4) at (6.4,10.0) {};
\node[purpleDot]    (r1w5) at (8.0,10.0) {};
\node[darkgreenDot] (r1w6) at (9.6,10.0) {};
\node[terminal]     (r1w8) at (11.2,10.0) {};

\node[lab,below=1.5mm of r1w3] {$\widetilde w_{1}^{(1)}$};
\node[lab,below=1.5mm of r1w7] {$\widetilde w_{2}^{(1)}$};
\node[lab,below=1.5mm of r1w2] {$\widetilde w_{3}^{(1)}$};
\node[lab,below=1.5mm of r1w1] {$\widetilde w_{4}^{(1)}$};
\node[lab,below=1.5mm of r1w4] {$\widetilde w_{5}^{(1)}$};
\node[lab,below=1.5mm of r1w5] {$\widetilde w_{6}^{(1)}$};
\node[lab,below=1.5mm of r1w6] {$\widetilde w_{7}^{(1)}$};
\node[lab,below=1.5mm of r1w8] {$w_{8}^{(1)}$};

\node[magentaDot]    (r2w2) at (0.0,6.9) {};
\node[purpleDot]     (r2w3) at (1.6,6.9) {};
\node[purpleDot]     (r2w5) at (3.2,6.9) {};
\node[lightgreenDot] (r2w7) at (4.8,6.9) {};
\node[magentaDot]    (r2w1) at (6.4,6.9) {};
\node[purpleDot]     (r2w4) at (8.0,6.9) {};
\node[lightgreenDot] (r2w6) at (9.6,6.9) {};
\node[terminal]      (r2w8) at (11.2,6.9) {};

\node[lab,below=1.5mm of r2w2] {$\widetilde w_{1}^{(2)}$};
\node[lab,below=1.5mm of r2w3] {$\widetilde w_{2}^{(2)}$};
\node[lab,below=1.5mm of r2w5] {$\widetilde w_{3}^{(2)}$};
\node[lab,below=1.5mm of r2w7] {$\widetilde w_{4}^{(2)}$};
\node[lab,below=1.5mm of r2w1] {$\widetilde w_{5}^{(2)}$};
\node[lab,below=1.5mm of r2w4] {$\widetilde w_{6}^{(2)}$};
\node[lab,below=1.5mm of r2w6] {$\widetilde w_{7}^{(2)}$};
\node[lab,below=1.5mm of r2w8] {$w_{8}^{(2)}$};

\node[darkgreenDot] (r3w7) at (0.0,3.8) {};
\node[purpleDot]    (r3w4) at (1.6,3.8) {};
\node[purpleDot]    (r3w5) at (3.2,3.8) {};
\node[blueDot]      (r3w2) at (4.8,3.8) {};
\node[darkgreenDot] (r3w6) at (6.4,3.8) {};
\node[blueDot]      (r3w1) at (8.0,3.8) {};
\node[purpleDot]    (r3w3) at (9.6,3.8) {};
\node[terminal]     (r3w8) at (11.2,3.8) {};

\node[lab,below=1.5mm of r3w7] {$\widetilde w_{1}^{(3)}$};
\node[lab,below=1.5mm of r3w4] {$\widetilde w_{2}^{(3)}$};
\node[lab,below=1.5mm of r3w5] {$\widetilde w_{3}^{(3)}$};
\node[lab,below=1.5mm of r3w2] {$\widetilde w_{4}^{(3)}$};
\node[lab,below=1.5mm of r3w6] {$\widetilde w_{5}^{(3)}$};
\node[lab,below=1.5mm of r3w1] {$\widetilde w_{6}^{(3)}$};
\node[lab,below=1.5mm of r3w3] {$\widetilde w_{7}^{(3)}$};
\node[lab,below=1.5mm of r3w8] {$w_{8}^{(3)}$};

\node[purpleDot]     (r4w5) at (0.0,0.7) {};
\node[lightgreenDot] (r4w6) at (1.6,0.7) {};
\node[blueDot]       (r4w1) at (3.2,0.7) {};
\node[purpleDot]     (r4w3) at (4.8,0.7) {};
\node[lightgreenDot] (r4w7) at (6.4,0.7) {};
\node[purpleDot]     (r4w4) at (8.0,0.7) {};
\node[blueDot]       (r4w2) at (9.6,0.7) {};
\node[terminal]      (r4w8) at (11.2,0.7) {};

\node[lab,below=1.5mm of r4w5] {$\widetilde w_{1}^{(4)}$};
\node[lab,below=1.5mm of r4w6] {$\widetilde w_{2}^{(4)}$};
\node[lab,below=1.5mm of r4w1] {$\widetilde w_{3}^{(4)}$};
\node[lab,below=1.5mm of r4w3] {$\widetilde w_{4}^{(4)}$};
\node[lab,below=1.5mm of r4w7] {$\widetilde w_{5}^{(4)}$};
\node[lab,below=1.5mm of r4w4] {$\widetilde w_{6}^{(4)}$};
\node[lab,below=1.5mm of r4w2] {$\widetilde w_{7}^{(4)}$};
\node[lab,below=1.5mm of r4w8] {$w_{8}^{(4)}$};

\begin{scope}[on background layer]
\draw[lightbluepath] (r1w3)--(r1w7)--(r1w2)--(r1w1)--(r1w4)--(r1w5)--(r1w6)--(r1w8);
\draw[greencurlypath] (r2w2)--(r2w3)--(r2w5)--(r2w7)--(r2w1)--(r2w4)--(r2w6)--(r2w8);
\draw[orangecurlypath] (r3w7)--(r3w4)--(r3w5)--(r3w2)--(r3w6)--(r3w1)--(r3w3)--(r3w8);
\draw[redcurlypath] (r4w5)--(r4w6)--(r4w1)--(r4w3)--(r4w7)--(r4w4)--(r4w2)--(r4w8);
\end{scope}

\node[magentaCommon]   (x1)  at (13.30,11.00) {};
\node[magentaCommon]   (x2)  at (13.30,9.70) {};
\node[blueCommon]      (x1p) at (13.30,8.40) {};
\node[blueCommon]      (x2p) at (13.30,7.10) {};
\node[purpleCommon]    (u1)  at (13.30,5.85) {};
\node[purpleCommon]    (u2)  at (13.30,4.65) {};
\node[purpleCommon]    (u3)  at (13.30,3.45) {};
\node[darkgreenCommon] (v1)  at (13.30,2.25) {};
\node[darkgreenCommon] (v2)  at (13.30,1.25) {};
\node[lightgreenCommon](v1p) at (13.30,0.15) {};
\node[lightgreenCommon](v2p) at (13.30,-0.90) {};

\node[magentaBox,anchor=west] at (13.75,11.00)
 {$w_{1}^{(1)}=w_{1}^{(2)}=x_{1}$};
\node[magentaBox,anchor=west] at (13.75,9.70)
 {$w_{2}^{(1)}=w_{2}^{(2)}=x_{2}$};

\node[blueBox,anchor=west] at (13.75,8.40)
 {$w_{1}^{(3)}=w_{1}^{(4)}=x_{1}'$};
\node[blueBox,anchor=west] at (13.75,7.10)
 {$w_{2}^{(3)}=w_{2}^{(4)}=x_{2}'$};

\node[purpleBox,anchor=west] at (13.75,5.85)
 {$w_{3}^{(1)}=w_{3}^{(2)}=w_{3}^{(3)}=w_{3}^{(4)}=u_{1}$};
\node[purpleBox,anchor=west] at (13.75,4.65)
 {$w_{4}^{(1)}=w_{4}^{(2)}=w_{4}^{(3)}=w_{4}^{(4)}=u_{2}$};
\node[purpleBox,anchor=west] at (13.75,3.45)
 {$w_{5}^{(1)}=w_{5}^{(2)}=w_{5}^{(3)}=w_{5}^{(4)}=u_{3}$};

\node[darkgreenBox,anchor=west] at (13.75,2.25)
 {$w_{6}^{(1)}=w_{6}^{(3)}=v_{1}$};
\node[darkgreenBox,anchor=west] at (13.75,1.25)
 {$w_{7}^{(1)}=w_{7}^{(3)}=v_{2}$};

\node[lightgreenBox,anchor=west] at (13.75,0.15)
 {$w_{6}^{(2)}=w_{6}^{(4)}=v_{1}'$};
\node[lightgreenBox,anchor=west] at (13.75,-0.90)
 {$w_{7}^{(2)}=w_{7}^{(4)}=v_{2}'$};

\begin{scope}[on background layer]

\draw[purpleEdge]    (r1w3) .. controls (3.0,8.4) and (9.6,6.4) .. (u1);
\draw[darkgreenEdge] (r1w7) .. controls (-1.6,10.7) and (-1.8,2.1) .. (v2);
\draw[magentaEdge]   (r1w2) .. controls (6.1,10.65) and (10.8,10.20) .. (x2);
\draw[magentaEdge]   (r1w1) .. controls (7.0,11.05) and (11.8,11.15) .. (x1);
\draw[magentaEdge]   (r1w4) .. controls (8.6,10.40) and (11.9,10.05) .. (x2);
\draw[darkgreenEdge] (r1w5) .. controls (10.3,8.4) and (11.4,3.1) .. (v1);
\draw[purpleEdge]    (r1w6) .. controls (11.7,8.7) and (12.25,6.25) .. (u1);

\draw[lightgreenEdge] (r2w7) .. controls (6.7,5.7) and (10.0,-0.25) .. (v2p);
\draw[lightgreenEdge] (r2w6) .. controls (10.2,5.4) and (10.4,0.7) .. (v1p);
\draw[magentaEdge]    (r2w2) .. controls (2.5,7.8) and (9.7,8.55) .. (x2);

\draw[blueEdge]      (r3w2) .. controls (8.0,4.8) and (11.7,8.1) .. (x1p);
\draw[blueEdge]      (r3w1) .. controls (10.3,4.4) and (12.1,7.25) .. (x2p);
\draw[darkgreenEdge] (r3w7) .. controls (1.0,2.35) and (9.8,1.45) .. (v2);

\draw[purpleEdge]    (r4w5) .. controls (4.0,0.95) and (10.1,2.55) .. (u3);

\end{scope}

\end{tikzpicture}
    \caption{The spanning tree from Lemma \ref{tree-d-equal2}}
    \label{fig-G-colored2-tree}
\end{figure}
}
\end{example}

\newpage

{\em Proof of Theorem \ref{CLT}:} Since $\sigma_R^2(t) \sim c(t) R^d$, it is enough to prove that
\[
\frac{F_R(t)}{R^{d/2}} \stackrel{d}{\longrightarrow} X\sim \cN\big(0,c(t)\big) \quad \mbox{as} \quad R \to \infty.
\]
We write
\[
\frac{F_R(t)}{R^{d/2}} =\frac{1}{R^{d/2}} \sum_{n=1}^N I_n\big(g_{n,R}(\cdot,t)\big)+\frac{1}{R^{d/2}} \sum_{n>N}  I_n\big(g_{n,R}(\cdot,t)\big).
\]
Let $X_N$ be a random variable with a $\cN\big(0,\sum_{n=1}^N c_n(t)\big)$ distribution. We show that:\\
(i) the first term converges in distribution as $R \to \infty$ to  $X_N$, for any $N\geq 1$;\\
(ii) the second term converges to $0$ in $L^2(\Omega)$ as $N\to \infty$, uniformly in $R>0$. 

Moreover, $X_N \stackrel{d}{\to} X$ as $N \to \infty$, since $c(t)=\sum_{n\geq 1}c_n(t)$. 
Our argument is summarized by the diagram below:
\begin{equation*}
\begin{tikzcd}[column sep=large, row sep=large,remember picture]
\frac{1}{R^{d/2}} \sum_{n=1}^N I_n\big(g_{n,R}(\cdot,t)\big) \arrow[r, "d", "R \to \infty"',] \arrow[d, "N \to \infty", "\text{in $L^2(\Omega)$, uniformly in } R"'] & X_N  \arrow[d, "N \to \infty","d"']  \\
\frac{1}{R^{d/2}} F_R(t) \arrow[r, dashrightarrow] & X
\end{tikzcd}
\end{equation*}

To prove (i), it is enough to show that for any $N \geq 1$,
\[
\frac{1}{R^{d/2}} \Big(I_1\big(g_{1,R}(\cdot,t)\big), \ldots, I_N\big(g_{N,R}(\cdot,t\big))  \Big) \stackrel{d}{\longrightarrow} \cN\big(\pmb{0}, \pmb{\Sigma}_N(t)\big) \quad \mbox{as} \quad R \to \infty,
\]
where $\pmb{\Sigma}_N(t)$ is a diagonal matrix with elements $c_1(t),\ldots,c_N(t)$ on the diagonal. For this, we use Theorem \ref{4th-theorem-m}. Condition {\bf (a)} clearly holds, since $I_1(g_{1,R}(\cdot,t)), \ldots, I_N(g_{N,R}(\cdot,t))$ are uncorrelated, and 
\[
\frac{1}{R^{d}} {\rm Var} \big(I_{n}(g_{n,R}(\cdot,t))\big) \to c_n(t) \quad \mbox{as} \quad R \to \infty,
\]
for any $n=1,\ldots,N$; see \eqref{sigma-nR}. Condition {\bf (b)} is also clear, since for any $n=1,\ldots,N$,
\[
\frac{1}{R^{2d}}\bE\big[I_{n,R}^4(g_{n,R}(\cdot,t))\big] \to 3 c_n^2(t) \quad \mbox{as} \quad R \to \infty,
\]
due to \eqref{4mom-I1} and the fact that $\sigma_{n,R}^4(t)\sim c_n^2(t) R^{2d}$ as $R \to \infty$.

\medskip

To prove (ii), we will show that
\begin{equation}
\label{CLT-ii}
\sup_{R>0} \frac{1}{R^d} \bE \Big|\sum_{n>N} I_n\big(g_{n,R}(\cdot,t)\big)\Big|^2 \to 0 \quad \mbox{as} \ N \to \infty.
\end{equation}
By orthogonality and \eqref{sigma-nR1},
\begin{align*}
\sup_{R>0}\frac{1}{R^d} \Big|\sum_{n>N} I_n\big(g_{n,R}(\cdot,t)\big)\Big|^2 & =\sup_{R>0}\frac{1}{R^d} \sum_{n>N} \sigma_{n,R}^2(t)=\sup_{R>0}\sum_{n>N} \int_{B_{2R}}  \rho_{t}^{(n)}(z)\frac{{\rm Vol}\big(B_R \cap B_R(-z)\big)}{R^d}dz
\\
& \leq \omega_d \sum_{n>N}\int_{\bR^d} \rho_{t,n}(z)dz=\sum_{n>N}c_n(t).
\end{align*}

Since $\sum_{n\geq 1}c_n(t)<\infty$, the last term converges to $0$ as $N\to \infty$. This proves \eqref{CLT-ii}, and concludes the proof of Theorem \ref{CLT}.

\appendix

\section{Fourth cumulant estimate}
\label{appA}

In this section, we give an estimate for the fourth cumulant of a random variable which lies in a fixed Poisson chaos space. We assume that $N$ is a Poisson random measure on an arbitrary measure space $(Z, {\mathcal Z}, \fm)$. We let $\cH=L^2(Z, {\mathcal Z}, \fm)$ and $\cH^{\odot p}$ be the set of symmetric functions in $\cH^{\otimes p}$.

\medskip

We recall some definitions. For any $f \in \cH^{\otimes p}$ and $g \in \cH^{\otimes q}$, we define the {\em modified contractions}, as follows:
\begin{itemize}
\item[(i)]
 $f\star^0_0 g = f\otimes g$ is the usual tensor product of $f$
and $g$;

\item[(ii)]
 for $1\leq r\leq p\wedge q$,   $f\star_r^0 g$ is a  function
on $Z^{p+q-r}$ given by:
\begin{align*}
& (f \star_r^0 g)(\gamma_1,\ldots,\gamma_r,x_1,\ldots,x_{p-r},x_1',\ldots,x_{q-r}')=\\
& \quad \quad \quad 
f(\gamma_1, \ldots, \gamma_r,  x_1, \ldots , x_{p-r})  g(\gamma_1, \ldots, \gamma_r, x_1', \ldots , x_{q-r}'),
\end{align*}

\item[(iii)]
for $1\leq \ell \leq r \leq p \wedge q$, $f\star_r^{\ell} g$ is  function
on $Z^{m+n-k-\ell}$, given by
\begin{align*}
& (f \star_r^{\ell} g) (\gamma_1,\ldots,\gamma_{r-\ell},x_1,\ldots,x_{p-r},x_1',\ldots,x_{q-r}')=\\
& \int_{Z^{\ell}} f(z_1,\ldots,z_{\ell},\gamma_1, ... , \gamma_{r-\ell},  x_1, \ldots, x_{p-r})  g(z_1,\ldots,z_{\ell},\gamma_1, \ldots, \gamma_{r-\ell}, x_1', \ldots , x_{q-r}')\fm(dz_1)\ldots \fm(dz_{\ell}),
\end{align*}
provided that the integral is finite.
\end{itemize}

If $1\leq \ell<r$, 
$f \star_r^{\ell} g$ may not exist. 
Moreover,  $f \star_r^r g=f \otimes_r g$ for any $1 \leq r \leq p \wedge q$, where
\begin{align*}
& (f \otimes_r g)(x_1,\ldots,x_{p-r},x_1',\ldots,x_{q-r}')=\\
& \quad \int_{Z^r} f(z_1,\ldots,z_r,x_1,\ldots,x_{p-r})
g(z_1,\ldots,z_r,x_1',\ldots,x_{q-r}')m(dz_1)\ldots m(dz_r).
\end{align*}

If $p=q$, then $f \star_p^p g=\langle f,g \rangle_{\cH^{\otimes p}}$. Moreover, $f \star_p^0 f=f^2$.

\medskip

We let 
$f \, \widetilde{\star}_r^{\ell} \, g$ be the symmetrization of $f \star_r^{\ell} g$, and $f \widetilde{\otimes} g$ be the symmetrization of $f \otimes g$. 

\begin{lemma}
\label{cumulant} 
Let $F=I_p(f)$ for some $f\in \cH^{\odot p}$ and $p\geq 1$, be such that $\bE(F^4)<\infty$ and $\sigma^2=\bE(F^2)>0$. Then
$\bE(F^4)-3\sigma^4$ is bounded above by a linear combination of terms of the form 
\[
\|f \star_r^{\ell} f\|_{\cH^{\otimes(2p-r-\ell)}}^2 \quad \mbox{with $0\leq \ell \leq r \leq p$ and $(\ell,r) \not=(p,p)$, $(\ell,r)\not=(0,0)$}.
\]
\end{lemma}

\begin{proof}
By Theorem 2.2 of \cite{DP18-ECP}, $F^2=\sum_{k=0}^{2p} I_{k}(h_{k})=\sum_{m=0}^{2p} I_{2p-m}(h_{2p-m})$, where
\begin{equation}
\label{def-hpm}
h_{2p-m}=\sum_{r=\lceil \frac{m}{2} \rceil}^{m \wedge p} r! \binom{p}{r}^2 \binom{r}{m-r} f \, \widetilde{\star}_r^{m-r} \, f,
\end{equation}
i.e. $I_k(h_k)$ is the projection of $F^2$ on the $k$-th Poisson chaos. Here $\lceil \cdot \rceil$ is the upper floor function: for any $x \in \bR$, $\lceil x \rceil \in \bZ$ and $\lceil x \rceil-1<x \leq \lceil x \rceil$. This means that in the sum \eqref{def-hpm}, the upper index $\ell:=m-r$ can be equal to $r$, since $r \geq \lceil \frac{m}{2} \rceil \geq \frac{m}{2}$. $\ell$ can also be equal to $0$, namely when $m\wedge p=m$ (i.e. $m\leq p$) and $r=m$. But $(\ell,r)$ cannot be $(0,0)$.

By relation (43) of \cite{DP18-AOP},
\begin{align*}
\bE(F^4)-\sigma^4 &=\sum_{k=1}^{2p-1}{\rm Var}\big(I_k(h_k)\big) +(2p)!\| f \widetilde{\otimes} f \|_{\cH^{\otimes 2p}}^2 = \sum_{k=1}^{2p-1}k! \|h_k\|_{\cH^{\otimes k}}^2 +(2p)!\| f \widetilde{\otimes} f \|_{\cH^{\otimes 2p}}^2.
\end{align*}
We now invoke the highly non-trivial relation (5.2.12) of \cite{NP12}, which says that
\[
(2p)!\| f \widetilde{\otimes} f \|_{\cH^{\otimes 2p}}^2=2(p!)^2 \| f \|_{\cH^{\otimes p}}^4+
(p!)^2 \sum_{r=1}^{p-1} \binom{p}{r}^2 \|f \otimes_r f\|_{\cH^{\otimes (2p-2r)}}^2.
\]
We obtain that:
\begin{align*}
\bE[F^4]-\sigma^4=2(p!)^2 \| f \|_{\cH^{\otimes p}}^4+
\sum_{m=1}^{2p-1}(2p-m)! \|h_{2p-m}\|_{\cH^{\otimes (2p-m)}}^2+(p!)^2 \sum_{r=1}^{p-1} \binom{p}{r}^2 \|f \otimes_r f\|_{\cH^{\otimes (2p-2r)}}^2.
\end{align*}
(This relation is similar to (3.5) of \cite{DP18-ECP}, in which $h_{2p-m}$ is replaced by $h_{2p-m}^{(n)}$ given by (3.3) {\em ibid.})
Since $\sigma^4=(\bE[F^2])^2=(p!)^2 \|f\|_{\cH^{\otimes p}}^4$, we obtain that
\[
\bE[F^4]-3\sigma^4=
\sum_{m=1}^{2p-1}(2p-m)! \|h_{2p-m}\|_{\cH^{\otimes (2p-m)}}^2+(p!)^2 \sum_{r=1}^{p-1} \binom{p}{r}^2 \|f \otimes_r f\|_{\cH^{\otimes (2p-2r)}}^2.
\]

Using definition \eqref{def-hpm} of $h_{2p-m}$, triangular inequality, the inequality $(\sum_{i=1}^N a_i)^2 \leq N \sum_{i=1}^N a_i^2$, and the fact that $\|\widetilde{g} \|_{\cH^{\otimes k}} \leq \|g \|_{\cH^{\otimes k}}$ for $g \in \cH^{\otimes k}$, it follows that
\begin{align*}
\|h_{2p-m}\|_{\cH^{\otimes (2p-m)}}^2 
& \leq p \sum_{r=\lceil \frac{m}{2} \rceil}^{m \wedge p} (r!)^2  \binom{p}{r}^4 \binom{r}{m-r}^2  \|f \, \star_r^{m-r} \, f\|_{\cH^{\otimes (2p-m)}}^2.
\end{align*} 
Hence,
\begin{align*}
\bE[F^4]-3\sigma^4 & \leq p \sum_{m=1}^{2p-1}(2p-m)! \sum_{r=\lceil \frac{m}{2} \rceil}^{m \wedge p} (r!)^2  \binom{p}{r}^4 \binom{r}{m-r}^2  \|f \, \star_r^{m-r} \, f\|_{\cH^{\otimes (2p-m)}}^2\\
& + (p!)^2 \sum_{r=1}^{p-1} \binom{p}{r}^2 \|f \, \star_r^{r} \, f \|_{\cH^{\otimes (2p-2r)}}^2
\end{align*}
Finally, we note that $(m-r,r) \not =(p,p)$, for any $m=1,\ldots,2p-1$. 

\end{proof}

\section{A graph-related integral}

In this section, we give an estimate for an integral associated with a multigraph, which is used in the proof of Theorem \ref{gg-lem}.
We first recall some basic elements of graph theory. 

\medskip

An (undirected) {\em graph} is a pair $G=(V,E)$, where $V$ is a set whose elements are called {\em vertices} and $E$ is set of unordered pairs $\{u,v\}$ of vertices, whose elements are called {\em edges.}  A {\em multigraph} is graph which is permitted to have multiple edges, i.e. edges which have the same end points. A graph is {\em connected} is for any two vertices $u,v \in V$, there exists a path whose endpoints are $u$ and $v$. A multigraph is {\em finite} if it has a finite number of vertices and a finite number of edges.

A {\em tree} is an (undirected) graph in which every pair of distinct vertices is connected by exactly one path. (A tree is a connected graph which does not have any cycles.)

A {\em spanning tree} of a graph $G$ is a tree which includes all vertices of $G$. Any finite connected multigraph  has a (non-empty) spanning tree. 
If  $G=(V,E)$ is a finite connected multigraph with at least two vertices, and $T=(V,E_T)$ is a spanning tree of $G$, then 
\[
|E_T|=|V|-1,
\]
where $|S|$ the cardinality of the set $S$.

A {\em leaf} is a vertex that has degree 1. (The {\em degree} of a vertex is the number of edges incident to that vertex.)  Every finite tree with at least two vertices has at least two leaves.

\begin{lemma}
\label{lemB} Let $G=(V,E)$ be a finite (non-empty) connected multigraph with $|V|\geq 2$. For any edge $e\in E$, we denote $e=\{u_e,v_e\}$ with $u_e,v_e \in V$. For any $e \in E$, let $s_{e}>0$ be arbitrary.  Fix an arbitrary vertex $p \in V$. Let $T=(V,E_T)$ be a spanning tree of $G$.

For  any $t>0$, let $g_t:\bR^d \to \bR_{+}$ be a symmetric measurable function such that $g_t \in L^{\infty}(\bR^d) \cap 
L^{1}(\bR^d)$. Then
\[
I:=\int_{(\bR^{d})^{|V|-1}} \prod_{e \in E} g_{s_{e}}(x_{v_e}-x_{u_e}) \prod_{v \in V \verb2\2 \{p\}} dx_v \leq \prod_{e \in E\verb2\2 E_T} \|g_{s_e}\|_{L^{\infty}(\bR^d)} \prod_{e \in E_T} \|g_{s_e}\|_{L^{1}(\bR^d)}.
\]
\end{lemma}

\begin{proof}
{\em Step 1.} We bound the factors which correspond to a non-tree edge by the $L^{\infty}$-norm:
\[
 \prod_{e \in E} g_{s_{e}}(x_{v_e}-x_{u_e}) \leq  \prod_{e \in E\verb2\2 E_T} \|g_{s_e}\|_{L^{\infty}(\bR^d)} 
 \prod_{e \in E_T} g_{s_{e}}(x_{v_e}-x_{u_e}).
\]
Hence,
$I \leq \prod_{e \in E\verb2\2 E_T} \|g_{s_e}\|_{L^{\infty}(\bR^d)} I_1$,
where
\[
I_1:=\int_{(\bR^d)^{|V|-1}} \prod_{e \in E_T} g_{s_{e}}(x_{v_e}-x_{u_e}) \prod_{v \in V \verb2\2 \{p\}} dx_v.
\]

{\em Step 2.} It remains to prove that
\[
I_1 \leq \prod_{e \in E_T} \|g_{s_e}\|_{L^{1}(\bR^d)}.
\]

We apply Fubini theorem to calculate the integral $I_1$. 
Pick any leaf $w_1 \not =p$. Let $e_1=\{u_{e_1},v_{e_{1}}\}$ be the edge incident to $w_1$. Say $w_1=u_{e_1}$. We integrate $dx_{w_1}$ first:
\[
\int_{\bR^d} g_{s_{e_1}}(x_{v_{e_1}}-x_{w_1})dx_{w_1}=\|g_{s_{e_1}}\|_{L^1(\bR^d)}.
\]
Hence,
\[
I_1=\|g_{s_{e_1}}\|_{L^1(\bR^d)}\int_{(\bR^{d})^{|V|-2}} \prod_{e \in E_T\verb2\2\{e_1\}} g_{s_{e}}(x_{v_e}-x_{u_e}) \prod_{v \in V \verb2\2 \{p,w_1\}} dx_v.
\]

Pick another leaf $w_1 \not \in \{p,w_1\}$. Let $e_2=\{u_{e_2},v_{e_{2}}\}$ be the edge incident to $w_2$. Say $w_2=u_{e_2}$. We apply Fubini theorem again by integrating $dx_{w_2}$:
\[
\int_{\bR^d} g_{s_{e_2}}(x_{v_{e_2}}-x_{w_2})dx_{w_2}=\|g_{s_{e_2}}\|_{L^1(\bR^d)}.
\]
We obtain:
\[
I_1=\|g_{s_{e_1}}\|_{L^1(\bR^d)} \|g_{s_{e_2}}\|_{L^1(\bR^d)} \int_{(\bR^{d})^{|V|-3}} \prod_{e \in E_T\verb2\2\{e_1,e_2\}} g_{s_{e}}(x_{v_e}-x_{u_e}) \prod_{v \in V \verb2\2 \{p,w_1,w_2\}} dx_v.
\]
We continue until we eliminate all leaves. If $|V|=n$, the last integral will be
\[
\int_{\bR^d}g_{s_{e_{n-1}}} (x_p-x_{w_{n-1}}) dx_{w_{n-1}}=\|g_{s_{e_{n-1}}}\|_{L^1(\bR^d)},
\]
where $e_{n-1}=\{p,w_{n-1}\} \in E_T$ is the last edge in the tree, connecting $p$ with the last leaf $w_{n-1}\not=p$.
\end{proof}

We consider an example to illustrate Lemma \ref{lemB}.

\begin{example}
{\rm Consider the multigraph $G=(V,E)$ with $V=\{1,2,\ldots,7\}$,  given by Figure \ref{graph2}, with $p=6$ and spanning tree $T$ with edges $E_T=\{e_1,e_2,e_6,e_8,e_{10},e_{11}\}$.

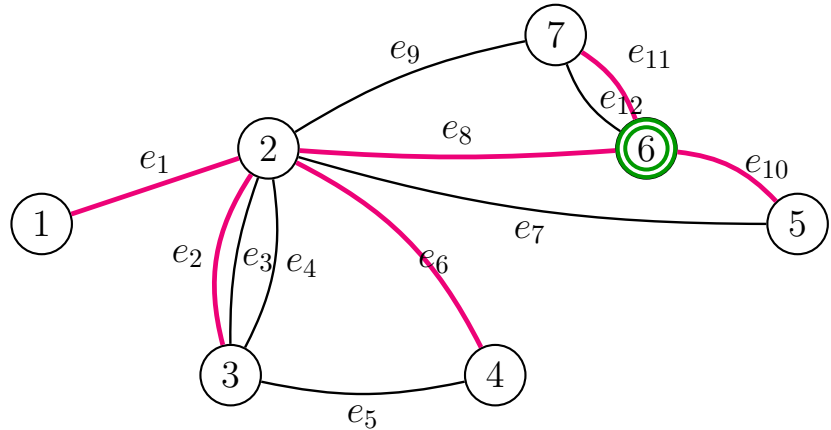
\begin{figure}[htbp]
    \centering

\begin{tikzpicture}[
    vertex/.style={
        circle,
        draw=black,
        line width=0.8pt,
        minimum size=8mm,
        inner sep=0pt,
        font=\large
    },
    edge/.style={
        draw=black,
        line width=0.9pt
    },
    highlighted/.style={
        draw=magenta!80!red,
        line width=1.8pt
    },
    every node/.style={font=\large}
]

\node[vertex] (v1) at (0,0) {1};
\node[vertex] (v2) at (3,1) {2};
\node[vertex] (v3) at (2.5,-2) {3};
\node[vertex] (v4) at (6,-2) {4};
\node[vertex] (v5) at (10,0) {5};
\node[vertex] (v6) at (8,1) {6};
\node[vertex] (v7) at (6.8,2.5) {7};

\draw[green!60!black, line width=1.5pt] (v6) circle [radius=0.28];
\draw[green!60!black, line width=1.5pt] (v6) circle [radius=0.38];


\draw[highlighted] (v1) -- node[above] {$e_1$} (v2);

\draw[highlighted] (v2) to[bend right=24] node[left] {$e_2$} (v3);
\draw[edge]        (v2) to[bend right=10] node[right,xshift=-2pt] {$e_3$} (v3);
\draw[edge]        (v2) to[bend left=18]  node[right] {$e_4$} (v3);

\draw[edge] (v3) to[bend right=12] node[below] {$e_5$} (v4);

\draw[highlighted] (v2) to[bend left=18] node[below right] {$e_6$} (v4);

\draw[edge] (v2) to[bend right=8] node[below] {$e_7$} (v5);

\draw[highlighted] (v2) to[bend right=4] node[above] {$e_8$} (v6);

\draw[edge] (v2) to[bend left=10] node[above] {$e_9$} (v7);

\draw[highlighted] (v6) to[bend left=20] node[right] {$e_{10}$} (v5);

\draw[highlighted] (v7) to[bend left=18] node[above right] {$e_{11}$} (v6);
\draw[edge]        (v7) to[bend right=18] node[right] {$e_{12}$} (v6);

\end{tikzpicture}
\caption{Multipgraph with $V=\{1,\ldots,7\}$, $p=6$ and $E_T=\{e_1,e_2,e_6,e_8,e_{10},e_{11}\}$.}
\label{graph2}
\end{figure}

We use boxes to exhibit the factors corresponding to the edges of $T$. In Step 1, {we drop the factors which are not in a box}:
\begin{align*}
I
&=
\int_{(\mathbb{R}^{d})^{6}}
\blackbox{g_{s_{e_1}}(x_1-x_2)}
\blackbox{g_{s_{e_2}}(x_2-x_3)}
g_{s_{e_3}}(x_2-x_3)
g_{s_{e_4}}(x_2-x_3)
g_{s_{e_5}}(x_3-x_4)
\blackbox{g_{s_{e_6}}(x_2-x_4)}
\\
&\qquad {}\times
g_{s_{e_7}}(x_2-x_5)
\blackbox{g_{s_{e_8}}(x_2-x_6)}
g_{s_{e_9}}(x_2-x_7)
\blackbox{g_{s_{e_{10}}}(x_5-x_6)}
\blackbox{g_{s_{e_{11}}}(x_6-x_7)}
g_{s_{e_{12}}}(x_6-x_7)
\\
&\qquad {}\times
dx_1\,dx_2\,dx_3\,dx_4\,dx_5\,dx_7  \\
& \leq \prod_{i \in \{3,4,5,7,9,12\}}\|g_{s_{e_i}}\|_{L^{\infty}(\bR^d)} I_1,
\end{align*}
where $I_1$ is the $dx_1\,dx_2\,dx_3\,dx_4\,dx_5\,dx_7$ integral of the product:
\begin{align*}
g_{s_{e_1}}(x_1-x_2)g_{s_{e_2}}(x_2-x_3)g_{s_{e_6}}(x_2-x_4) g_{s_{e_8}}(x_2-x_6) g_{s_{e_{10}}}(x_5-x_6) g_{s_{e_{11}}}(x_6-x_7).
\end{align*}

In Step 2, to compute $I_1$, we choose a variable among $x_1,x_2,x_3,x_4,x_5,x_7$ which appears {\em only once} in this product. This corresponds to choosing the first leaf. Our choices are $1,3,5,7$. Say we pick $x_5$. Then we integrate
\[
\int_{\bR^d}g_{s_{e_{10}}}(x_5-x_6)  dx_5=\|g_{s_{e_{10}}}\|_{L^1(\bR^d)}.
\]
We continue in this manner. Say that the order that we pick for the integration is $x_5, x_7,x_1,x_3,x_4$, which means that we choose leaves 5,7,1,3,4.
The last integral is
\[
\int_{\bR^d}g_{s_8}(x_2-x_6)dx_2=\|g_{s_8}\|_{L^{\infty}(\bR^d)}.
\]
Although the order in which we pick the leaves is not unique, at the end we obtain:
\[
I_1 =\prod_{i \in \{1,2,6,8,10,11\} } \|g_{s_{e_i}} \|_{L^1(\bR^d)}.
\]

}
\end{example}

{\em Acknowledgement.} The authors would like to thank Guangqu Zheng for useful comments.

\end{document}